\documentclass[11pt, a4paper]{article}

\usepackage{mathtools, amsthm, accents, tikz, amssymb, latexsym, amsmath, amscd, amsfonts, array, lmodern, enumerate, stmaryrd, rotating, caption, graphicx, hyperref, float,tikz-cd, color, yfonts, stmaryrd, bookmark}

\usepackage{authblk}

\usepackage[all]{xy}
\CompileMatrices
\hypersetup{nesting=true,debug=true,naturalnames=true}

\def\Xnpq{\protect\operatorname{X}(n,p\times q)}
\def\UX{\protect\operatorname{UX}}
\def\UXnpq{\protect\operatorname{UX}(n,p\times q)}
\def\C{\protect\operatorname{C}}
\def\UC{\protect\operatorname{UC}}
\def\puz{\protect\operatorname{Puz}}

\def\cell{\protect\operatorname{cell}(n,2)}
\def\ucell{\protect\operatorname{ucell}(n,2)}

\def\barra{\,\raisebox{-1.8mm}{\rule{.2mm}{5mm}}\,}

\def\UF{\protect\operatorname{UConf}}
\def\DF{\protect\operatorname{DConf}}
\def\UDF{\protect\operatorname{UDConf}}

\def\zcl{\protect\operatorname{zcl}}

\newtheorem{proposition}{Proposition}[section]
\newtheorem{corollary}[proposition]{Corollary}
\newtheorem{definition}[proposition]{Definition}
\newtheorem{theorem}[proposition]{Theorem}
\newtheorem{remark}[proposition]{Remark}
\newtheorem{example}[proposition]{Example}
\newtheorem{lemma}[proposition]{Lemma}
\newtheorem{examples}[proposition]{Examples}

\newtheorem{notation}[proposition]{Notation}

\begin{document}

\title{A combinatorial genesis of the right-angled relations in \\ Artin's classical braid groups}

\author{Omar Alvarado-Gardu\~{n}o$^*$, Jes\'us Gonz\'alez\footnote{Cinvestav, Departamento de Matem\'aticas, M\'exico City 07000, M\'exico, {\tt oalvarado@math.cinvestav.mx} and 
{\tt jesus.glz-espino@cinvestav.mx}} \ and Matthew Kahle\footnote{Department of Mathematics, 
The Ohio State University, Columbus, OH 43210-1174, USA, {\tt mkahle@math.osu.edu}}}

\date{}

\maketitle

\begin{abstract}
The configuration space $\UC(n,p\times q)$ of $n$ unlabelled non-overlapping unit squares in a $p\times q$ rectangle is known to recover the homotopy type of the classical configuration space of $n$ unlabelled points in the plane, provided $\min\{p,q\}\geq n$. Thus the fundamental group $B_n(p\times q)$ of $\UC(n,p\times q)$ yields a $(p,q)$-approximation of Artin's classical braid group $B_n$. We describe a right-angled Artin group presentation for $B_n(p\times q)$ in cases where $\UC(n,p\times q)$ is known to be aspherical. When $\min\{p,q\}=2$, our presentation agrees with Artin's classical presentation for $B_n$ removing the Artin-Tits relations. This allows us to deduce the value of the Lusternik-Schnirelmann category of the corresponding aspherical spaces $\UC(n,p\times q)$, as well as the values of all their $k$-sequential topological complexities, both in the classical (Rudyak \emph{et al.})~and distributional (Dranishnikov \emph{et al.})~contexts.
\end{abstract}



\tableofcontents

\section{Introduction}
While the idea of braiding can be traced back to Hurwitz's 1891 paper \cite{hur91} on ramified coverings of surfaces, it was only until 1925 that the notion was formalized and studied in detail. In his seminal work~\cite{MR3069440}, Emil Artin showed that, for a positive integer $n$, the braid group $B_n$ on $n$-strands can be presented by $n-1$ generators $x_1, x_2,\ldots,x_{n-1}$ subject to the following two types of relations:
\begin{align}
x_ix_j&=x_jx_i, \text{ \ for $|i-j|>1$, and}\label{rarelations}\\
x_ix_{i+1}x_i&=x_{i+1}x_ix_{i+1}, \text{ \ for $1\leq i\leq n-2$.}\nonumber
\end{align}
The first type of relations, which we refer to as the RA relations, constitute the right-angled foundation of $B_n$, while to second type, which we refer to as the AT relations, give $B_n$ its Artin-Tits braiding essence. AT relations are responsible for deep connections of $B_n$ in theoretical physics via the Yang-Baxter relations. On the other hand, RA relations (and their extension to arbitrary graphs) provide a rich source of complex phenomena in geometric group theory. Given the independent relevance of these two types of relations, we aim at isolating each of them in a topological way. In this paper we address the task of isolating the RA relations.

\smallskip
For positive integers\footnote{We assume $n,p,q\geq2$ in order to avoid trivial situations.} $n$, $p$, $q$, let $\UC(n,p\times q)$ denote the configuration space of $n$ unlabelled non-overlapping unit squares in a rectangle of size $p\times q$. The homotopy colimit determined by the obvious inclusions $\UC(n,p\times q)\hookrightarrow\UC(n,(p+1)\times q)$ is known to agree with the configuration space of $n$ unlabelled non-overlapping unit disks in an infinite planar strip of width $q$, denoted  by $\UC(n,\infty\times q)$ or simply $\UC(n,q)$. Likewise, the union over all $q\geq2$ of the spaces $\UC(n,q)$ is known to be homotopy equivalent to the classical unordered configuration space $\UF(\mathbb{R}^2,n)$ of $n$ points in the plane. Indeed, inclusion maps yield homotopy equivalences
\begin{equation}\label{aproximacion}
\UC(n,p\times q)\simeq\UC(n,q) \text{ for } p\geq n \text{ \ \ and \ \ }\UC(n,q)\simeq\UF(\mathbb{R}^2,n) \text{ for } q\geq n.
\end{equation}
See \cite[bottom of page 2596]{MR4640135} for the first equivalence, and \cite[Theorems 1.2, 3.1, and the observations before Proposition~3.7]{MR4298668} for the second one. In particular the fundamental groups $B_n(p\times q):=\pi_1(\UC(n,p\times q))$ approximate $B_n$ as $p$ and $q$ increase. Unfortunately the algebraic approximation recovers the topological approximation only under limited conditions, as only a few of the spaces $\UC(n,p\times q)$ retain the asphericity property of $\UF(\mathbb{R}^2,n)$, see \cite[Theorem~3.2]{aspher} and \cite[bottom of page 374]{MR4298668}. Actually, discarding the spaces $\UC(n,p\times q)$ with $\min\{p,q\}\geq n$, whose asphericity comes directly from (\ref{aproximacion}), the only spaces $\UC(n,p\times q)$ known to be aspherical have either
\begin{equation}\label{casosaesfericos}
 n\in\{pq-1,pq-2\} \quad\text{ or }\quad \min\{p,q\}=2 \text{ \,with } \max\{p,q\}\geq n.
 \end{equation}
As explained later in the paper (Remark \ref{explicacion}), asphericity comes essentially from Abrams' Ph.D.~thesis \cite{MR2701024} in the first case, while (\ref{aproximacion}) and \cite[Theorem 3.6]{MR4298668} yield the assertion in the second case.

\begin{remark}\label{conceivable}{\em
It is conceivable that $\UC(n,p\times 2)$ would be aspherical for $p<n$ too. However, to the authors' best knowledge, such an interesting possibility stands as an intriguing open question.  
}\end{remark}

With this preparation, we give a bird's-eye view of our main achievements. At the foremost, we characterize most of the (aspherical) topologies summarized in (\ref{casosaesfericos}). This is attained through a precise description of the corresponding fundamental group $B_n(p\times q)$. In all cases addressed we obtain a right-angled Artin group presentation bearing some degree of resemblance with the RA-relations defining $B_n$. The resemblance is complete in cases satisfying the second condition in (\ref{casosaesfericos}). In fact, when $\min\{p,q\}=2$, we show that $B_n(p\times q)$ can be presented with generators $x_1,\ldots,x_{n-1}$ subject only to the RA-relations~(\ref{rarelations}) provided $\max\{p,q\}\geq\min\{n,\frac{n+5}2\}$. The latter inequality is a substantial relaxation of the corresponding one in the second condition of (\ref{casosaesfericos}). On the other hand, for the aspherical topologies in (\ref{casosaesfericos}) not addressed in this work, namely those with $n=pq-2$ and $\min\{p,q\}\geq3$, we show that $B_n(p\times q)$ is in fact the braid group of index 2 associated to the grid-type graph $\Gamma_{p,q}$ described in Subsection \ref{subsectionFSA}. It is worth noticing that the characterization of graph braid groups of index 2 admitting a right-angled Artin group structure is an open hard problem (see \cite[Section 5]{MR2833585}).

\smallskip
The following notation prepares a detailed description of aspherical topologies satisfying the first restriction in (\ref{casosaesfericos}). For a nonnegative integer $k$, we let $B(k)$ stand for the bipartite graph with vertices $u_i$ and $v_i$ for $1\leq i\leq k$, and an edge joining $u_i$ and $v_j$ whenever $i\leq j$. Additionally, for a graph $\Gamma$, we let $k+\Gamma$ stand for the graph obtained by adding $k$ isolated vertices to $\Gamma$. For instance $0+\Gamma=\Gamma$, while $B(0)$ is the empty graph.

\begin{theorem}\label{1hueco}
For $p,q\geq2$, $B_{pq-1}(p\times q)$ is a free group of rank $(p-1)(q-1)$. Indeed, $\UC(pq-1,p\times q)$ is homotopy equivalent to a wedge of $(p-1)(q-1)$ circles.
\end{theorem}

\begin{theorem}\label{2huecos}
For $p\geq3$, $B_{2p-2}(p\times2)$ is the right-angled Artin group determined by the graph 
$3+B(p-3)$.
\end{theorem}

The case $p=2$ not included in Theorem \ref{2huecos} is special as $\UC(2,2\times2)\simeq S^1$, in view of~(\ref{aproximacion}), so that $B_2(2\times2)\cong\mathbb{Z}$.

Recall that the classifying space of the right-angled Artin group determined by a simple graph $\Gamma$ is homotopy equivalent to the polyhedral power $(S^1)^{\text{Flag}(\Gamma)}$, where $\text{Flag}(\Gamma)$ stands for the flag complex associated to $\Gamma$. The flag construction is superfluous in Theorem \ref{2huecos}, as $3+B(p-3)$ is bipartite. For instance, the cases $p\in\{3,4\}$ in Theorem \ref{2huecos} assert that $\UC(6,4\times2)\simeq(S^1\times S^1)\vee \UC(4,3\times2)$ with $\UC(4,3\times2)$ homotopy equivalent to $\bigvee_3S^1$.

\begin{remark}\label{noraag}{\em
According to \cite[Conjecture 5.1 and Figure 29]{MR2833585} and Proposition \ref{dualidad} below, $B_{pq-2}(p\times q)$ is not expected to admit a right-angled Artin group structure if $\min\{p,q\}\geq3$. Instead, the type of groups that might be arising here are illustrated by \cite[Theorem 2.5 and Examples~2.4]{Omar}, where it is shown that, as long as $p\geq3$, $B_{3p-2}(p\times3)$ is an HNN extension of a certain ``squared'' version of $B_{2p-2}(p\times 2)$. 
}\end{remark}

Our description of the aspherical topologies corresponding to the second condition in~(\ref{casosaesfericos}) is contained in Theorem \ref{combgen} below, which provides the announced combinatorial genesis for the RA relations in~$B_n$. Indeed, the relevant groups retain both the generators and the RA relations of $B_n$, while AT relations are absent.

\begin{theorem}\label{combgen}
For $n\geq2$, and either $n\leq p$ or $n\leq 2p-5$, $B_n(p\times 2)$ is the right-angled Artin group generated by $n-1$ elements $x_1,\ldots,x_{n-1}$ subject to the RA-type relations~(\ref{rarelations}).
\end{theorem}

Theorems \ref{1hueco}, \ref{2huecos} and \ref{combgen} provide a right-angled Artin group presentation for the fundamental group of any space $U(n,p\times q)$ with $\min\{p,q\}=2$ (i.e.~``of height 2'') when the number $pq-n$ of ``empty squares"  is different from $3$ and $4$. A few of the height-2 groups $B_n(p\times q)$ with $pq-n\in\{3,4\}$ also admit a right angled Artin group presentation (see Examples \ref{e1} below), though the defining graph is not well-behaved. This has prevented us from giving a general presentation for all height-2 cases with 3 or 4 empty squares.

\smallskip
As illustrated in Remark \ref{noraag}, primitive forms of the blend of RA and AT relations characterizing $B_n$ must start arising in $B_n(p\times q)$ for $\min(p,q)\geq3$ (i.e.~in height at least 3). Yet, interesting new fundamental groups (most likely coming from non-aspherical topologies) are expected to hold only for $\max\{p,q\}<n$ because, as observed in \cite[bottom of page 374]{MR4298668}, $B_n(q):=B_n(\infty\times q)\cong B_n$ for $q\geq3$.

\smallskip
Our methods can also be used to prove the fact that all inclusion maps of the form $\UC(n,p\times2)\hookrightarrow\UC(n,(p+1)\times2)$ among spaces relevant to Theorem \ref{combgen} induce conjugation-type isomorphisms in fundamental groups.

\begin{examples}\label{e1}{\em 
Our methods show that $B_5(4\times2)$ ---with 3 empty squares--- is the right-angled Artin group associated to the graph with vertices $x_1,x_2,x_3,x_4$ and 2 edges, the first one joining $x_1$ and $x_3$, and the second one joining $x_2$ and $x_4$. This yields the correct generators, but a commutator relation is missing with respect to (\ref{rarelations}), namely, the one associated to the missing commutativity relation $x_1x_4\neq x_4x_1$. Likewise, we can prove that $B_7(5\times2)$ ---also with 3 empty squares--- is generated by elements $x_i$ ($1\leq i\leq 6$), subject exclusively to the relations
\begin{align*}
x_1 & \mbox{ commutes with $x_6x_2^{-1}$ and with $x_6x_4^{-1}$;}\\
x_2 & \mbox{ commutes with $x_3^{-1}x_6$, with $x_4^{-1}x_6$ and with $x_5^{-1}x_6$;}\\
x_3 & \mbox{ commutes with $x_6x_4^{-1}$;}\\
x_4 & \mbox{ commutes with $x_5^{-1}x_6$.}
\end{align*}
By comparing these relations to the ones summarized in (\ref{codificados}) later in the paper, i.e., those amounting to the right-angled Artin group structure in Theorem \ref{combgen} when $p\leq n\leq 2p-5$, it becomes clear that the number of missing commutativity relations in this example is even larger than in the case of $B_5(4\times2)$. This is actually the first situation for which we do not know whether a right-angled Artin group structure arises. Lastly, our methods show that $B_6(5\times2)$ ---with 4 empty squares--- has a right-angled Artin group presentation fitting Theorem~\ref{combgen}. Further calculations suggest that this successful example would be an exceptional instance among spaces with 4 empty squares.
}\end{examples}

\section{LS-category and (distributional) sequential TC}
The results in the previous section allow us to give a full description of the Lusternik-Schnirelmann (LS) category of the aspherical spaces $\UC(n,p\times q)$ in Theorems \ref{1hueco}, \ref{2huecos} and \ref{combgen}, as well as of all their sequential topological complexities, both in the classical (Rudyak \emph{et al.})~and distributional (Dranishnikov \emph{et al.})~contexts. For background details on LS-category, on sequential topological complexity, and on distributional sequential topological complexity, the reader is referred to \cite{MR1990857}, \cite{MR3331610,MR2593704}, and \cite{jauhari-daundkar, MR4863030}, respectively.

\smallskip
For a space $X$ and an integer $r\geq2$, let $\zcl_{\mathbb{Q},r}(X)$, $\text{TC}_r(X)$, and $\text{dTC}_r(X)$ denote the rational $r$-th zero-divisor cup-length of $X$, the $r$-th sequential topological complexity of $X$, and the $r$-th distributional sequential topological complexity of $X$, respecively. These non-negative integers are related through the inequalities $\zcl_{\mathbb{Q},r}(X)\leq\text{dTC}_r(X)\leq\text{TC}_r(X)$. In particular, the three numbers must agree provided $\zcl_{\mathbb{Q},r}(X)=\text{TC}_r(X)$. The later equality is known to hold when $X=K(\text{RAAG}(\Gamma),1)$ is an Eilenberg-MacLane space for the right angled Artin group $\text{RAAG}(\Gamma)$ determined by a simplicial graph $\Gamma$. The actual values of $\text{TC}_r(K(\text{RAAG}(\Gamma),1))$ and of the LS-category $\text{cat}(K(\text{RAAG}(\Gamma),1))$, were first described in \cite{MR3642766} and \cite{MR2457428}, respectively ---see also \cite{MR4430946,MR3924508}. In both cases the answer is given in terms of the structure of cliques of $\Gamma$.

\begin{proposition}\label{cliques}
For $X=K(\emph{RAAG}(\Gamma),1)$,
\begin{itemize}
\item $\emph{cat}(X)$ is the largest cardinality of cliques of $\Gamma$.
\item $\emph{TC}_r(X)$ is the maximal total cardinality $\sum_{i=1}^r |C_i|$ of cliques $C_1,\ldots,C_r$ of \,$\Gamma$ with total empty intersection, $\bigcap_{i=1}^rC_i=\varnothing$.
\end{itemize}
\end{proposition}
Empty cliques are allowed in the second instance of Proposition \ref{cliques}. In view of~(\ref{casosaesfericos}), Theorems \ref{1hueco}, \ref{2huecos} and \ref{combgen} then yield:

\begin{corollary}\label{secatsinartincases}
For $\min\{p,q\}\geq2$,
\begin{enumerate}
\item $\emph{cat}(\UC(pq-1,p\times q))=1$, and 
$$
\emph{TC}_r(\UC(pq-1,p\times q))=\emph{dTC}_r(\UC(pq-1,p\times q))=\begin{cases}
r-1, & p=q=2; \\
r, & \text{otherwise.}
\end{cases}
$$
\item $\emph{cat}(\UC(2p-2,p\times2))=\min\{2,\left\lfloor\frac{p}2\right\rfloor\}$, and
$$
\emph{TC}_r(\UC(2p-2,p\times2))=\emph{dTC}_r(\UC(2p-2,p\times2))=\begin{cases}
r-1, & p=2; \\
r, & p=3; \\
2r-1, & p=4; \\
2r, & p\geq5.
\end{cases}
$$
\item For $p\geq n=2m+\varepsilon$ with $m\geq1$ and $\varepsilon\in\{0,1\}$, $\emph{cat}(\UC(n,p\times2))=m$, and
$$
\emph{TC}_r(\UC(n,p\times2))=\emph{dTC}_r(\UC(n,p\times2))=rm-1+\varepsilon
$$\end{enumerate}
\end{corollary}

\begin{proof}
All spaces treated here are Eilenberg-MacLane spaces of right angled Artin groups, so it suffices to describe the relevant cliques in each case. For instance, non-empty cliques are singletons in part 1, as well as in part 2 with $2\leq p\leq3$. This yields the relevant assertions.

For part 2 with $p\geq4$: Since the graph $B(p-3)$ is bipartite, maximal cliques have cardinality either 1 or 2. If $p=4$, there is a single maximal clique of cardinality 2, namely $\{u_1,v_1\}$, which can be taken as $C_1=\cdots=C_{r-1}$ in Proposition \ref{cliques}, with $C_r$ being any other singleton. For $p\geq5$, take $C_1,\ldots, C_{r-1}$ as before, and $C_r=\{u_2,v_2\}$. Again, this yields the relevant asertions.

Lastly, in part 3 take $C_1=\cdots=C_{r-1}=\{1,3,5,\ldots,2m-1\}$, which is the clique of maximal possible cardinality. Then take $C_r=\{2,4,\cdots,2m-2\}$ when $n$ is even, and $C_r=\{2,4,\cdots,2m\}$ when $n$ is odd.
\end{proof}

Some of the assertions in part 3 of Corollary \ref{secatsinartincases} recover results from N.~Wawrykow, who used a Salvetti-type homotopy model for $\UC(n,2)$ slightly different from the polyhedral-product model used in the proof of Proposition \ref{cliques}. Namely, our LS-category assertion is essentially contained in \cite[Theorem 1.1 and Proposition 3.2]{MR4840250}, while our (d)TC$_r$-assertions for $n$ odd recover \cite[Proposition~4.1]{wawstrip}. Our (d)TC$_r$-assertions for an even $n$ are new.

\section{Preliminaries}\label{prelims}

\subsection{Discrete Morse-Poincar\'e approach to the fundamental group}\label{dmtatofg}
The calculations in this paper are based on Forman's discrete Morse theory. We review the methods and refer the reader to \cite{MR2171804,MR1358614} for details.

\smallskip
Let $\mathcal{F}$ denote the face poset of a connected finite regular CW complex $X$. Here $\mathcal{F}$ is partially ordered by inclusion of closures of cells. For a cell $a\in\mathcal{F}$, we write $a^{(p)}$ to indicate that $a$ is $p$-dimensional. Consider the Hasse diagram $H_\mathcal{F}$ of $\mathcal{F}$ as a directed graph with arrows $a^{(p+1)}\searrow b^{(p)}$ oriented from the higher dimensional cell $a$ to its codimension-1 face $b$. A partial matching $W$ on $H_\mathcal{F}$ is a directed subgraph of $H_\mathcal{F}$ all whose vertices have degree 1, while the corresponding $W$-modified Hasse diagram $H_\mathcal{F}(W)$ is obtained from $H_\mathcal{F}$ by reversing all arrows of $W$. It is standard to use the notation $b^{(p)}\nearrow a^{(p+1)}$ for such a reversed edge, calling $a$ and $b$ the $W$-collapsible and $W$-redundant cells, respectively, of the $W$-pair $(a,b)$. We then focus on $W$-gradient paths, i.e., zigzag paths $\lambda$ in $H_\mathcal{F}(W)$
$$
a_0\nearrow b_1\searrow a_1\nearrow b_2\searrow\cdots\nearrow b_k\searrow a_k.
$$
Note that the condition $a_i\neq a_{i+1}$ is forced by construction. We say that such a $W$-gradient path $\lambda$ is a cycle when $a_0=a_k$ (in which case $k\geq2$ must hold). We say that $W$ is a gradient field if it has no cycles. In such a case, cells of $X$ that are neither redundant nor collapsible are said to be $W$-critical.

\smallskip
Let $X$ be as above and assume $W$ is a gradient field on $X$ with a single $W$-critical 0-cell $v_0$. As shown in \cite[Proposition 2.3]{MR2171804}, $0$-cells and $W$-collapsible 1-cells of $X$ assemble a maximal tree $T_{X,W}$ of the 1-skeleton of $X$. The ``Morse-Poincar\'e'' presentation for $\pi_1(X,v_0)$ in \cite[Theorem~2.5]{MR2171804} can be summarized as follows. First, collapse $T_{X,W}$ to $v_0$ to get a generating set $\{\beta_e\}_{e}$ of $\pi_1(X,v_0)$, where $e$ runs over the set of $W$-critical 1-cells of $X$. Relations are then given by the \emph{reduced} form of boundary words of $W$-critical 2-cells. In order to spell this out, start by orienting edges of $T_{X,W}$ so they point away from $v_0$, while any other 1-cell of $X$ is oriented arbitrary. For each 0-cell $u$, let $\beta_u$ denote the unique oriented simplicial path in $T_{X,W}$ from $v_0$ to $u$. Then, for any oriented 1-cell $e$ from $u_1$ to $u_2$, the loop
\begin{equation}\label{representingloop}
\beta_{u_1}\star e\star\beta_{u_2}^{-1}
\end{equation}
represents a homotopy class $\beta_e\in\pi_1(X;v_0)$. In what follows we denote $\beta_e$ simply by $e\in\pi_1(X;v_0)$, letting the context clarify whether we refer to the actual $W$-critical 1-cell or to the homotopy class so defined (we will mostly be interested in the latter one). In these terms, a generating set of $\pi_1(X,v_0)$ is given by the set of (homotopy classes of) $W$-critical 1-cells $e$. In order to define the relations among these generators, a reduction process is constructed for elements $z\in\pi_1(X,v_0)$ expressed as words in the alphabet given by oriented 1-cells and their inverses. The actual element of $\pi_1(X,v_0)$ does not change throughout the process, but its expression as a word gets standardized. Explicitly, in a typical step of the process one chooses a letter $w_i$ of a given word $w=w_1w_2\cdots w_k$ representing $z$, and takes the following action:
\begin{itemize}
\item If $w_i$ is a collapsible 1-cell or its inverse, then we simply delete $w_i$ from $w_1w_2\cdots w_k$. This corresponds to the fact that $T_{X,W}$ has been collapsed down to $v_0$.
\item If $w_i$ is a redundant 1-cell, then we replace the letter $w_i$ in $w_1w_2\cdots w_k$ by the subword
$$
w_{i,\ell_i}^{-\epsilon_{i,\ell_i}}w_{i,\ell_{i-1}}^{-\epsilon_{i,\ell_{i-1}}}\cdots w_{i,1}^{-\epsilon_{i,1}}.
$$
Here, as depicted in Figure \ref{coherent}, we have oriented the $W$-pair $W(w_i)$ of $w_i$ coherently to $w_i$, forming then the boundary word $$w_i w_{i,1}^{\epsilon_{i,1}} w_{i,2}^{\epsilon_{i,2}} w_{i,3}^{\epsilon_{i,3}}\cdots w_{i,\ell_i}^{\epsilon_{i,\ell_i}}$$ for $W(w_i)$, where the exponent $\epsilon_{i,j}\in\{\pm1\}$ of $w_{i,j}$ is 1 if and only if the chosen orientation of $w_{i,j}$ is coherent with that of $W(w_i)$. The resulting substitution thus bookkeeps the simple collapse in Figure \ref{coherent}.
\item Accordingly, if $w_i$ is the inverse of a redundant 1-cell, then we replace the letter $w_i$ in $w_1w_2\cdots w_k$ by the subword
$$
w_{i,1}^{\epsilon_{i,1}}w_{i,2}^{\epsilon_{i,2}}\cdots w_{i,\ell_i}^{\epsilon_{i,\ell_i}}.
$$
\end{itemize}

\begin{figure}
\centering
\begin{tikzpicture}[x=.6cm,y=.6cm]
\draw (0,0) ellipse (55pt and 55pt);
\draw[very thick,->](3.225,.1)--(3.225,.1);
\draw(3.05,0)--(1.2,0);
\draw[->](-1.2,0)--(-3,0);
\draw[->](3,.3)--(-2.3,1.9);
\draw[->](3,-.3)--(-2.3,-1.9);
\draw[->](2.95,.6)--(-.5,2.95);
\draw[->](2.95,-.6)--(-.5,-2.95);
\draw[->](2.85,1)--(2,2.25);
\draw[->](2.85,-1)--(2,-2.25);
\node at (0,3.7) {};
\node at (0,0) {$W(w_i)$};
\node at (3.7,0) {\,\,\,\,$w_i$};
\draw[very thick] (2.93,-1.34) arc (-25:25:1.9cm);
\draw[very thick] (3.01,1.41)--(2.82,1.32);
\draw[very thick] (3.01,-1.41)--(2.82,-1.32);

\draw (10,0) ellipse (55pt and 55pt);
\draw[very thick,->](13.225,.1)--(13.225,.1);
\node at (10,3.7) {};
\node at (10,0) {$W(w_i)$};
\node at (13.7,0) {\,\,\,\,$w_i$};
\node at (12.65,2.8) {$w_{i,1}$};
\node at (7.3,2.85) {$w_{i,2}$};
\node at (7.1,-3) {$w_{i,\ell_i-1}$};
\node at (6.4,.1) {$\vdots$};
\node at (12.65,-3.05) {$w_{i,\ell_i}$};
\draw[very thin] (10,0) ellipse (22pt and 22pt);
\draw[->] (11.289,0)--(11.289,0);
\draw[very thick] (12.93,-1.34) arc (-25:25:1.9cm);
\draw[very thick] (13.01,1.41)--(12.82,1.32);
\draw[very thick] (13.01,-1.41)--(12.82,-1.32);
\draw[very thick] (10,3.07)--(10,3.33);
\draw[very thick] (10,-3.09)--(10,-3.35);
\draw[very thick] (7.17,1.31)--(6.97,1.4);
\draw[very thick] (7.17,-1.31)--(6.97,-1.4);
\end{tikzpicture}
\caption{The simple collapse for a $W$-pair and the corresponding step in the reduction process}
\label{coherent}
\end{figure}

Starting with some boundary word of each $W$-critical 2-cell $\alpha$, the process is iterated eliminating along the way every 2-letter subword of the form $e\cdot e^{-1}$ and $e^{-1}\cdot e$. As proved in \cite[Proposition~2.4]{MR2171804}, the process ends after a finite number of steps yielding a word $w_\alpha$ on $W$-critical 1-cells or their inverses. Furthermore, the resulting word $w_\alpha$ depends only on $\alpha$ (and, of course, on $W$), but not on the different possible choices made along the actual process. The set of relations for the Morse-Poincar\'e presentation of $\pi_1(X,v_0)$ is $\{w_\alpha\,\colon\alpha\text{ is a $W$-critical 2-cell}\}$.

\subsection{Discrete configuration spaces}\label{subsectionUDF}
The reader is referred to \cite{MR2701024,MR2171804} for details and proof arguments of the facts reviewed in this and the next subsection. 

For a graph $\Gamma$ with cells $e$, i.e., vertices and (closed) edges of $\Gamma$, Abrams defined the discrete configuration space $\DF(\Gamma,n)$ of $n$ non-colliding labelled cells of $\Gamma$ as the cubical subcomplex of $\Gamma^n$ consisting of the cells $e_1 \times\cdots\times e_n$ satisfying $e_i \cap e_j = \varnothing$ for $i\neq j$. We use the standard notation $(e_1,\ldots,e_n) := e_1 \times\cdots\times e_n$. The corresponding unordered discrete configuration space $\UDF(\Gamma,n)$ is the quotient cubical complex of $\DF(\Gamma,n)$ by the cubical free action of the symmetric group $\Sigma_n$ on $n$ letters that permutes cell factors. The cell of $\UDF(\Gamma,n)$ given by the $\Sigma_n$-orbit of $(e_1,\ldots, e_n)$ will be denoted by $\{e_1,\ldots, e_n\}$.

\begin{remark}\label{extension}{\em
Abrams' construction of discrete configuration spaces on graphs, both in the ordered and unordered setting, applies word for word for more general cell complexes. Moreover, when $K$ is a cubical complex, the discrete configuration space $\DF(K,n)$ is a cubical subcomplex of $K^n$, while the orbit space $\UDF(K,n)$ inherits a natural cubical structure. In this context, if $L$ is a subcomplex of $K$ then, by construction, $\DF(L,n)$ (respectively, $\UDF(L,n)$) is a subcomplex of $\DF(K,n)$ (respectively, $\UDF(K,n)$).
}\end{remark}

In his Ph.D.~thesis \cite{MR2701024}, Abrams gave two arguments showing a remarkable property of discrete configuration spaces on graphs:

\begin{theorem}[Abrams, 2000]\label{Eilenbergmaclane}
Both $\DF(\Gamma,n)$ and $\UDF(\Gamma,n)$ are aspherical. (The first space might fail to be path-connected, in which case all of its components are aspherical.)
\end{theorem}

Discrete configuration spaces are interesting in several respects. For instance, for a graph~$\Gamma$, $\UDF(\Gamma,n)$ is known to recover the homotopy type of the classical configuration space $\UF(|\Gamma|,n)$ of the topological realization $|\Gamma|$, provided~$\Gamma$ satisfies a mild subdivision condition that depends on $n$, see \cite{MR2701024,MR3276733}. For instance, when $n=2$, the subdivision requirement is satisfied by all simple graphs. On the other hand, and more important for our purposes, discrete configuration spaces satisfy the duality property in Proposition~\ref{dualidad} below. This property might be known to experts (it was hinted to the authors by an anonymous referee for an earlier version of this work), though we have not been able to find an explicit statement in the literature. We provide proof details for completeness.

\begin{proposition}\label{dualidad}
Assume $\Gamma$ is a simple graph with $m$ vertices. For $1\leq r\leq m-1$, there is an isomorphism $\UDF(\Gamma,r)\cong\UDF(\Gamma,m-r)$ of cube complexes. (Note that $\UDF(\Gamma,m)$ is a single point.)
\end{proposition}
\begin{proof}
Let $V$ stand for the set of vertices of $\Gamma$. A $d$-dimensional cube $c$ of $\UDF(\Gamma,r)$ has the form $c=\{e_1,\ldots,e_d,u_1,\ldots,u_{r-d}\}$, where $d\leq \min\{r,m-r\}$, $u_1,\ldots,u_{r-d}\in V$, and each $e_i$ is an edge, say joining vertices $e_{i,1}$ and $e_{i,2}$. Since the set of vertices $$W_c:=V\setminus\left(\{e_{i,1},e_{i,2}\,|\,1\leq i\leq d\}\cup\rule{0mm}{4mm}\{u_j\,|\,1\leq j\leq r-d\}\right)$$ has cardinality $m-r-d$, we get a $d$-dimensional cube $\varphi(c):=\{e_1,\ldots,e_d\}\cup W_c$ of $\UDF(\Gamma,m-r)$. This construction is reversible (associating a $d$-dimensional cube $\varphi(c')$ of $\UDF(\Gamma,r)$ to each $d$-dimensional cube $c'$ of $\UDF(\Gamma,m-r)$), and sets a bijection between the $d$-dimensional cubes of $\UDF(\Gamma,r)$ and those of $\UDF(\Gamma,m-r)$. An easy check (see Figure \ref{facestructure}) then shows that this bijection preserves face structures.
\end{proof}

\begin{figure}[h!]
\centering
\begin{tikzpicture}[x=.8cm,y=.8cm]
\draw (0,0)--(1,0);
\node at (0,0) {$\bullet$};\node at (1,0) {$\bullet$};
\draw (0,1)--(1,1);
\node at (0,1) {$\bullet$};\node at (1,1) {$\bullet$};
\draw (0,2)--(1,2);
\node at (0,2) {$\bullet$};\node at (1,2) {$\bullet$};
\draw (0,3)--(1,3);
\node at (0,3) {$\bullet$};\node at (1,3) {$\bullet$};
\node at (0.5,0.26) {\footnotesize$e_d$};
\node at (0.5,1.26) {\footnotesize$\ldots$};
\node at (0.5,2.26) {\footnotesize$e_2$};
\node at (0.5,3.26) {\footnotesize$e_1$};
\draw (-1.7,1.5) ellipse (20pt and 30pt);
\draw (2.7,1.5) ellipse (20pt and 30pt);
\node at (-1.7,1.5) {\footnotesize$W_c$};
\node at (2.7,1.5) {\footnotesize$U_c$};

\draw[thick, dotted] (0+10,0)--(1+10,0);
\node at (0+10,0) {$\bullet$};\node at (1+10,0) {$\bullet$};
\draw (0+10,1)--(1+10,1);
\node at (0+10,1) {$\bullet$};\node at (1+10,1) {$\bullet$};
\draw (0+10,2)--(1+10,2);
\node at (0+10,2) {$\bullet$};\node at (1+10,2) {$\bullet$};
\draw (0+10,3)--(1+10,3);
\node at (0+10,3) {$\bullet$};\node at (1+10,3) {$\bullet$};
\node at (0.5+10,0.26) {\footnotesize$e_d$};
\node at (0.5+10,1.26) {\footnotesize$\ldots$};
\node at (0.5+10,2.26) {\footnotesize$e_2$};
\node at (0.5+10,3.26) {\footnotesize$e_1$};
\draw (-1.7+10,1.5) ellipse (20pt and 30pt);
\draw (2.7+10,1.5) ellipse (20pt and 30pt);
\node at (-1.7+10,1.5) {\footnotesize$W_c$};
\node at (2.7+10,1.5) {\footnotesize$U_c$};
\node at (11,-.26) {\footnotesize$u$};
\node at (10,-.26) {\footnotesize$w$};

\draw [->] (11.2,-.26) to [out=0,in=210] (12.03,.36);
\draw [->] (9.74,-.26) to [out=180,in=340] (9.01,.36);

\end{tikzpicture}
\caption{Left: The $d$-cell $c=\{e_1,\ldots,e_d\}\cup U_c$ determines and is determined by the cell $\varphi(c)=\{e_1,\ldots,e_d\}\cup W_c$. Here $U_c=\{u_1,\ldots,u_{r-d}\}$ as in the proof, and $U_c$, $W_c$ and the end points of the edges $e_i$ partition $V$. Right: The face of $c$ obtained after replacing $e_d$ by $u$ corresponds to the face of $\varphi(c)$ obtained after replacing $e_d$ by $w$.}
\label{facestructure}
\end{figure}

\subsection{Farley-Sabalka gradient field}\label{subsectionFSA}
Fix positive integers $p$ and $q$. In this work we focus attention on the graph $\Gamma_{p,q}$ given by the restriction of the integer grid $\mathbb{Z}\times\mathbb{Z}$ to the rectangle $[1,p]\times[1,q]$. The left hand-side in Figure~\ref{maxtree} illustrates the case $(p,q)=(6,4)$. We assess the topology of $\UDF(\Gamma_{p,q},n)$ and related complexes via the gradient field $W_{p,q,n}$ constructed in \cite{MR2171804} by Farley and Sabalka. A detailed description of $W_{p,q,n}$ follows.

\smallskip
Let $T_{p,q}$ be the horizontal zigzag maximal tree of $\Gamma_{p,q}$, as illustrated on the right hand-side of Figure \ref{maxtree} for $(p,q)=(6,4)$. The linearity of $T_{p,q}$ yields on the nose a linear ordering of vertices of $\Gamma_{p,q}$, which we refer to as the $T_{p,q}$-ordering. Vertices of $\Gamma_{p,q}$ will then be denoted by their assigned number, while an edge of $\Gamma_{p,q}$ joining vertex $i$ to vertex $j$ with $i<j$ will be denoted by $[i,j]$, orienting it from $i$ to $j$. Note that edges of $T_{p,q}$ have a natural linear order too, say by assigning the number $i$ to the edge $[i-1,i]$. Such ordering will be referred to as the $T_{p,q}$-ordering of edges of $T_{p,q}$. On the other hand, edges outside $T_{p,q}$, referred to as deleted edges and illustrated by the vertical dotted lines on the right hand-side of Figure \ref{maxtree}, have the form $[i,j]$, for a unique $j$ with $i+1<j$. In such a case, we will use the notation $e_i:=[i,j]$. Note that $e_i$ makes sense for integers $i$ that are smaller than $p(q-1)$ and non-$p$-divisible.

\begin{figure}
\centering
\begin{tikzpicture}[x=.8cm,y=.8cm]
\draw (0,0)--(5,0);
\node at (0,0) {$\bullet$};\node at (1,0) {$\bullet$};
\node at (2,0) {$\bullet$};\node at (3,0) {$\bullet$};
\node at (4,0) {$\bullet$};\node at (5,0) {$\bullet$};
\draw (0,1)--(5,1);
\node at (0,1) {$\bullet$};\node at (1,1) {$\bullet$};
\node at (2,1) {$\bullet$};\node at (3,1) {$\bullet$};
\node at (4,1) {$\bullet$};\node at (5,1) {$\bullet$};
\draw (0,2)--(5,2);
\node at (0,2) {$\bullet$};\node at (1,2) {$\bullet$};
\node at (2,2) {$\bullet$};\node at (3,2) {$\bullet$};
\node at (4,2) {$\bullet$};\node at (5,2) {$\bullet$};
\draw (0,3)--(5,3);
\node at (0,3) {$\bullet$};\node at (1,3) {$\bullet$};
\node at (2,3) {$\bullet$};\node at (3,3) {$\bullet$};
\node at (4,3) {$\bullet$};\node at (5,3) {$\bullet$};
\draw (0,0)--(0,3);\draw (1,0)--(1,3);\draw (2,0)--(2,3);
\draw (3,0)--(3,3);\draw (4,0)--(4,3);\draw (5,0)--(5,3);
\node at (0,-.5) {\footnotesize$1$};
\node at (1,-.5) {\footnotesize$2$};
\node at (2,-.5) {\footnotesize$3$};
\node at (3,-.5) {\footnotesize$4$};
\node at (4,-.5) {\footnotesize$5$};
\node at (5,-.5) {\footnotesize$6$};
\node at (-.5,0) {\footnotesize$1$};
\node at (-.5,1) {\footnotesize$2$};
\node at (-.5,2) {\footnotesize$3$};
\node at (-.5,3) {\footnotesize$4$};

\draw[very thick] (10,0)--(15,0);
\node at (10,0) {$\bullet$};\node at (11,0) {$\bullet$};
\node at (12,0) {$\bullet$};\node at (13,0) {$\bullet$};
\node at (14,0) {$\bullet$};\node at (15,0) {$\bullet$};
\draw[very thick] (10,1)--(15,1);
\node at (10,1) {$\bullet$};\node at (11,1) {$\bullet$};
\node at (12,1) {$\bullet$};\node at (13,1) {$\bullet$};
\node at (14,1) {$\bullet$};\node at (15,1) {$\bullet$};
\draw[very thick] (10,2)--(15,2);
\node at (10,2) {$\bullet$};\node at (11,2) {$\bullet$};
\node at (12,2) {$\bullet$};\node at (13,2) {$\bullet$};
\node at (14,2) {$\bullet$};\node at (15,2) {$\bullet$};
\draw[very thick] (10,3)--(15,3);
\node at (10,3) {$\bullet$};\node at (11,3) {$\bullet$};
\node at (12,3) {$\bullet$};\node at (13,3) {$\bullet$};
\node at (14,3) {$\bullet$};\node at (15,3) {$\bullet$};
\draw[thick,dotted] (10,0)--(10,3);
\draw[thick,dotted] (11,0)--(11,3);
\draw[thick,dotted] (12,0)--(12,3);
\draw[thick,dotted] (13,0)--(13,3);
\draw[thick,dotted] (14,0)--(14,3);
\draw[thick,dotted] (15,0)--(15,3);
\draw[very thick] (15,0)--(15,1);\draw[very thick] (15,2)--(15,3);
\draw[very thick] (10,1)--(10,2);
\node at (10,-.3) {\footnotesize$1$};
\node at (11,-.3) {\footnotesize$2$};
\node at (12,-.3) {\footnotesize$3$};
\node at (13,-.3) {\footnotesize$4$};
\node at (14,-.3) {\footnotesize$5$};
\node at (15,-.3) {\footnotesize$6$};
\node at (15.25,1) {\footnotesize$7$};
\node at (14.23,1.25) {\footnotesize$8$};
\node at (13.23,1.25) {\footnotesize$9$};
\node at (12.27,1.25) {\footnotesize$10$};
\node at (11.27,1.25) {\footnotesize$11$};
\node at (9.6,1) {\footnotesize$12$};
\node at (9.6,2) {\footnotesize$13$};
\node at (10.67,2.25) {\footnotesize$14$};
\node at (11.67,2.25) {\footnotesize$15$};
\node at (12.67,2.25) {\footnotesize$16$};
\node at (13.67,2.25) {\footnotesize$17$};
\node at (15.35,2) {\footnotesize$18$};
\node at (15.35,3) {\footnotesize$19$};
\node at (14,3.35) {\footnotesize$20$};
\node at (13,3.35) {\footnotesize$21$};
\node at (12,3.35) {\footnotesize$22$};
\node at (11,3.35) {\footnotesize$23$};
\node at (10,3.35) {\footnotesize$24$};
\end{tikzpicture}
\caption{Grid $\Gamma_{6,4}$ (left) and its maximal tree (right)}
\label{maxtree}
\end{figure}

Let $a=\{a_1,\ldots,a_n\}$ be a cell of $\UDF(\Gamma_{p,q},n)$ with ``ingredients'' $a_k$ ($1\leq k\leq n$). Thus, each $a_k$ is either a vertex $i$ or an edge $[i,j]$. A vertex ingredient $a_k$ of $a$ is said to be blocked in $a$ if either $a_k=1$ or, else, if $[a_k-1,a_k]$ intersects some cell $a_\ell$ with $\ell\neq k$. An edge ingredient $a_k$ of $a$ is said to be order disrespectful if $a_k=e_i$ for some vertex $i$. Blocked vertices and order-disrespectful edges of $a$ are said to be the critical ingredients of $a$.

\smallskip
Farley-Sabalka's gradient field $W_{p,q,n}$ on $\UDF(\Gamma_{p,q},n)$ then works as follows. A cell $a$ is $W_{p,q,n}$-critical if all its ingredients are critical. Otherwise, let $N(a)$ stand for the set of non-critical ingredients of a, and note that the $T_{p,q}$-orderings of vertices and edges of $T_{p,q}$ can be merged into a single linear ordering when restricted to $N(a)$. In these terms, $a$ is:
\begin{itemize}
\item[\emph{(i)}] $W_{p,q,n}$-redundant if the first element in $N(a)$ is a vertex ingredient $a_k=i$. The $W_{p,q,n}$-pair of $a$ is then obtained by replacing the vertex ingredient $a_k$ of $a$ by the edge ingredient $[i-1,i]$.
\item[\emph{(ii)}] $W_{p,q,n}$-collapsible if the first element in $N(a)$ is an edge ingredient $a_k=[i-1,i]$. The $W_{p,q,n}$-pair of $a$ is then obtained by replacing the edge ingredient $a_k$ of $a$ by the vertex ingredient $i$.
\end{itemize}

There is at most a single $W_{p,q,n}$-critical 0-dimensional cell, namely, the set of vertices $v_0:=\{1,2,\cdots,n\}$. In particular, $\UDF(\Gamma_{p,q},n)$ is either empty, when $n>pq$, or path connected, otherwise. Likewise, there are no positive dimensional $W_{p,q,n}$-critical cells when $\min(p,q)=1$, in which case $\UDF(\Gamma_{p,q},n)$ is contractible. So from now on we avoid trivial cases by implicitly assuming $2\leq\min(p,q)$ and $n<pq$.

\section{Discrete homotopy models}\label{homomodels}
\subsection{Disks on an infinite strip}\label{doais}
We study the topology of $\UC(n,2)$ through the homotopy equivalent complex $\ucell$ constructed by Alpert et al. We review the construction, referring the reader to \cite{MR4298668} for details. (We will have no need of the fact that there is a corresponding model for each $\UC(n,\omega)$.) 

\smallskip
The ordered configuration space $\C(n,2)$ is homotopy equivalent to a regular cell complex $\cell$ of dimension $\lfloor n/2 \rfloor$. Cells of $\cell$ are cubes, though it might happen that the intersection of two cubes is the union of several common faces. Cells of $\cell$ are indexed by  permutations $\sigma=(\sigma_1,\ldots,\sigma_n)\in\Sigma_n$ decorated by vertical bars separating some pairs of consecutive $\sigma$-values. Vertical bars are placed so that if no bar separates $\sigma_i$ from $\sigma_{i+1}$ for $i+1<n$, then a bar must separate $\sigma_{i+1}$ from $\sigma_{i+2}$. Thus, a given configuration of bars separates $\sigma$ into blocks as
$$
\left(
\sigma_{1,1},\ldots,\sigma_{\ell_1,1}\barra
\sigma_{1,2},\ldots,\sigma_{\ell_2,2}\barra\cdots\barra
\sigma_{1,k},\ldots,\sigma_{\ell_k,k}\right),
$$
where the $j$th block has size $\ell_j\leq2$. The dimension of the corresponding cell is $n-k$ or, equivalently, the sum of the dimensions of the blocks, where a size-$\ell$ block is declared to have dimension $\ell-1$. A cell of dimension $d$ is face of a cell of dimension $d+1$ provided the larger dimensional cell can be obtained from the smaller one by the operation of removing a bar and merging the two adjacent blocks by a shuffle. Since $\cell$ is regular, the given face condition fully characterizes attaching maps.

\smallskip
For a permutation $\sigma\in\Sigma_n$ and an ($n-1$)-tuple $\epsilon=(\epsilon_1,\ldots,\epsilon_{n-1})$ of zeros and ones so that $\epsilon$ has no consecutive zeros, we use the notation $e_{\sigma,\epsilon}$ for the cell of $\cell$ determined by the permutation $\sigma$ with bars placed so that there is a bar in between $\sigma_i$ and $\sigma_{i+1}$ if and only if $\epsilon_i=1$. In these terms, $\Sigma_n$ acts freely and cellularly on $\cell$ via $\tau\cdot e_{\sigma,\epsilon}:=e_{\tau\sigma,\epsilon}$. The corresponding orbit cell in $\ucell$ is denoted by $e_\epsilon$. The homotopy equivalence $\C(n,2)\simeq\cell$ in \cite{MR4298668} is $\Sigma_n$-equivariant, so the quotient complex $\ucell:=\cell/\Sigma_n$ is homotopy equivalent to $\UC(n,2)$.

\smallskip
As it will become apparent from the proof of Theorem \ref{combgen}, the cell structure of $\ucell$ is highly economical. A price to pay, though, is that $\ucell$ is no longer regular; for instance, it has a single $0$-dimensional cell but $n-1$ cells of dimension 1. The 0-dimensional cell is $e_{\epsilon(1)}$, where $\epsilon(1):=(1,1,\ldots,1)$, whereas the 1-dimensional cells have the form $e_{0_i}$ for $i\in\{1,\ldots,n-1\}$. Here $0_i:=(1,\ldots,1,0,1,\ldots,1)$, with the zero located in coordinate~$i$. Note that the end points of a representative $e_{\sigma,0_i}$ of $e_{0_i}$ are $e_{\sigma,\epsilon(1)}$ and $e_{\sigma\cdot\tau_i,\epsilon(1)}$, where $\tau_i$ is the transposition of $i$ and $i+1$. We agree to orient $e_{\sigma,0_i}$ from $e_{\sigma,\epsilon(1)}$ to $e_{\sigma\cdot\tau_i,\epsilon(1)}$. Such a choosing is $\Sigma_n$-equivariant and determines an orientation of $e_{0_i}$.

\smallskip
Likewise, 2-cells in $\ucell$ have the form $e_{0_{i,j}}$ for $i,j\in\{1,2,\ldots,n-1\}$ with
\begin{equation}\label{RAreason}
i+1<j.
\end{equation}
Here $0_{i,j}:=(1,1,\ldots,1,0,1,1,\ldots,1,0,1,1,\ldots,1)$, where the first zero is located in coordinate $i$ and the second zero is located in coordinate $j$. The best way to keep track of orientations in the pieces of the boundary of $e_{0_{i,j}}$ is in terms of some given representative. Namely, the boundary of
$$
e_{\sigma,0_{i,j}}=\left(\sigma_1\barra\cdots\barra\sigma_i,\sigma_{i+1}\barra\cdots\barra\sigma_j,\sigma_{j+1}\barra\cdots\barra\sigma_n\right)
$$
consists of the four 1-dimensional faces
\begin{itemize}
\item $e_{\sigma,0_i}=\left(\sigma_1\barra\cdots\barra\sigma_i,\sigma_{i+1}\barra\cdots\barra\sigma_j\barra\sigma_{j+1}\barra\cdots\barra\sigma_n\right)$, oriented from $e_{\sigma,\epsilon(1)}$ to $e_{\sigma\tau_i,\epsilon(1)}$;
\item $e_{\sigma\tau_j,0_i}=\left(\sigma_1\barra\cdots\barra\sigma_i,\sigma_{i+1}\barra\cdots\barra\sigma_{j+1}\barra\sigma_j\barra\cdots\barra\sigma_n\right)$, oriented from $e_{\sigma\tau_j,\epsilon(1)}$ to $e_{\sigma\tau_j\tau_i,\epsilon(1)}$;
\item $e_{\sigma,0_j}=\left(\sigma_1\barra\cdots\barra\sigma_i\barra\sigma_{i+1}\barra\cdots\barra\sigma_j,\sigma_{j+1}\barra\cdots\barra\sigma_n\right)$, oriented from $e_{\sigma,\epsilon(1)}$ to $e_{\sigma\tau_j,\epsilon(1)}$;
\item $e_{\sigma\tau_i,0_j}=\left(\sigma_1\barra\cdots\barra\sigma_{i+1}\barra\sigma_i\barra\cdots\barra\sigma_j,\sigma_{j+1}\barra\cdots\barra\sigma_n\right)$, oriented from $e_{\sigma\tau_i,\epsilon(1)}$ to $e_{\sigma\tau_i\tau_j,\epsilon(1)}$.
\end{itemize}
Note that $\tau_i\tau_j=\tau_j\tau_i$, in view of (\ref{RAreason}), so $e_{\sigma,0_{i,j}}$ is attached forming the square

\begin{equation}\label{rdc}
\begin{tikzpicture}[x=.6cm,y=.6cm]
\draw (0,0)--(0,4)--(4,4)--(4,0)--(0,0);
\node at (0,0) {$\bullet$};\node at (4,0) {$\bullet$};
\node at (0,4) {$\bullet$};\node at (4,4) {$\bullet$};
\node at (-1,-.2) {\tiny$e_{\sigma\tau_j,\epsilon(1)}$};
\node at (5.25,-.2) {\tiny$e_{\sigma\tau_i\tau_j,\epsilon(1)}$};
\node at (-.8,4.2) {\tiny$e_{\sigma,\epsilon(1)}$};
\node at (5,4.15) {\tiny$e_{\sigma\tau_i,\epsilon(1)}$};
\draw[thick,->] (0,2)--(0,1.9);\draw[thick,->] (4,2)--(4,1.9);
\draw[thick,->] (2,0)--(2.1,0);\draw[thick,->] (2,4)--(2.1,4);
\node at (2,-.5) {$e_{\sigma\tau_j,0_i}$};
\node at (2,4.5) {$e_{\sigma,0_i}$};
\node at (-.9,2) {$e_{\sigma,0_j}$};
\node at (5.2,2) {$e_{\sigma\tau_i,0_j}$};
\end{tikzpicture}
\end{equation}

\begin{proof}[Proof of Theorem \ref{combgen} for $p\geq n\geq 2$] In view of the first equivalence in (\ref{aproximacion}), it suffices to consider the case $p=\infty$. Consider the trivial (empty) gradient field on $\ucell$. As described in Subsection \ref{dmtatofg}, $B_n(2):=B_n(\infty\times2)$ has generators $e_{0_i}$ for $1\leq i\leq n-1$ subject to the relations given by the boundary words of the 2-cells $e_{0_{i,j}}$ for $i,j\in\{1,2,\ldots,n-1\}$ with $i+1<j$. Since the gradient field in use is empty, each word is already in its reduced form. Moreover, as it can be seen from (\ref{rdc}) by passing to the quotient, the boundary word determined by $e_{0_{i,j}}$ is given by the commutator of $e_{0_i}$ and $e_{0_j}$. The result follows.
\end{proof}

\subsection{Squares in a rectangle}\label{siar}
Let $\puz_{p,q}$ be the cubical complex obtained by restricting to $[1,p]\times[1,q]$ the canonical 2-dimensional cube-complex structure of the plane. Thus, the graph $\Gamma_{p,q}$ in Subsection~\ref{subsectionFSA} is the 1-dimensional skeleton of $\puz_{p,q}$. Following \cite{MR4640135}, the discrete configuration space $\DF(\puz_{p,q},n)$ will be denoted by $\Xnpq$, while $\UDF(\puz_{p,q},n)$ will be denoted by $\UXnpq$. Alpert et al showed that $\Xnpq$ sits naturally inside $\C(n,p\times q)$ as a strong deformation retract. Their proof argument is $\Sigma_n$-equivariant and, by passing to the quotient, we get:

\begin{proposition}\label{UXmodel}
$\UXnpq$ is homotopy equivalent to $\UC(n,p\times q)$.
\end{proposition}

From this point on we work with the discrete model $\UXnpq$, rather than with the actual space $\UC(n,p\times q)$. In particular, $B_n(p\times q)$ will now stand for the fundamental group of $\UXnpq$.

\smallskip
As observed in Remark \ref{extension}, $\UDF(\Gamma_{p,q},n)$ is a subcomplex of $\UXnpq$, so Farley-Sabalka's gradient field $W_{p,q,n}$ on the former complex is also a gradient field on the latter complex. Redundant and collapsible cells in the larger complex are described in Subsection~\ref{subsectionFSA}, while critical cells are those of the smaller complex together with all the cells outside $\UDF(\Gamma_{p,q},n)$. Neither the cell structure on $\UXnpq$ nor the corresponding gradient field $W_{p,q,n}$ are as efficient as those in the previous subsection; however, they are sufficiently explicit to permit a complete analysis of the fundamental groups in Therems~\ref{1hueco}, \ref{2huecos}, and \ref{combgen}. 

\begin{remark}\label{explicacion}{\em
A cell $c=\{e_1,\ldots,e_n\}$ of $\UXnpq$ lies outside $\UDF(\Gamma_{p,q},n)$ if and only if one of the ingredients $e_i$ is a square of $\puz_{p,q}$, in which case $\dim(c)\geq2$. Consequently $\UDF(\Gamma_{p,q},n)$ and $\UXnpq$ share 1-skeleta, so that the fact that $\UXnpq$ is path-connected follows from the corresponding property for $\UDF(\Gamma_{p,q},n)$ (\cite[Theorem~2.6]{MR2701024}). In fact, $\UDF(\Gamma_{p,q},n)=\UXnpq$ when $n\geq pq-2$, for then there is no room for a square-type ingredient. In particular, Theorem~\ref{Eilenbergmaclane} gives the asphericity assertion in the case of the first condition in (\ref{casosaesfericos}).
}\end{remark}

\begin{proof}[Proof of Theorem \ref{1hueco}]
Theorem \ref{UXmodel}, Remark \ref{explicacion}, and Proposition \ref{dualidad} yield
$$
\UC(pq-1,p\times q)\simeq\UX(pq-1,p\times q)=\UDF(\Gamma_{p,q},pq-1)\cong\UDF(\Gamma_{p,q},1)=\Gamma_{p,q}.
$$
The result follows since the Euler characteristic of $\Gamma_{p,q}$ is $1-(p-1)(q-1)$.
\end{proof}

We assume $2\leq n\leq pq-2$ throughout the rest of the paper.

\section{\texorpdfstring{A raw presentation for $B_n(p\times q)$ via Farley-Sabalka's field}{Raw presentation}}\label{rawpresentation}

Recall that the $W_{p,q,n}$-critical $d$-cells of $\UF(\Gamma_{p,q},n)$ are given by cardinality-$n$ sets $c$ consisting of $d$ deleted edges of $\Gamma_{p,q}$ and $n-d$ vertices of $\Gamma_{p,q}$ which are blocked in $c$. The same description applies for $W_{p,q,n}$-critical $d$-cells of $\UXnpq$ except that, in addition, all cells $c=\{c_1,\ldots,c_n\}$ of $\UXnpq$ outside $\UF(\Gamma_{p,q},n)$, i.e., where some $c_i$ is a square (so $d\geq2$), are $W_{p,q,n}$-critical too.

\smallskip
As reviewed in Subsection \ref{subsectionFSA}, a 1-cell $c=\{c_1,\ldots,c_n\}$ (either of $\UDF(\Gamma_{p,q},n)$ or $\UXnpq$, as both complexes share 1-skeleton), say with edge ingredient $c_1=[i,j]$ (so $i<j$), is $W_{p,q,n}$-collapsible if and only if $j=i+1$ ---i.e., $[i,j]$ is an edge of $T_{p,q}$--- and all vertex ingredients $k$ in $c$ with $k<i$ are blocked in $c$. In such a case, the $W_{p,q,n}$-pair of $c$ is the 0-cell $\{i+1,c_2,\dots,c_n\}$. On the other hand, there are two types of $W_{p,q,n}$-redundant 1-cells $c=\{c_1,\ldots,c_n\}$:
\begin{itemize}
\item those with a deleted-edge ingredient and (at least) a vertex ingredient, say $c_1=k$, which is unblocked in $c$, and 
\item those with a non-deleted-edge ingredient $[i,i+1]$ and (at least) a vertex ingredient, say $c_1=k$ with $k<i$, which is unblocked in $c$.
\end{itemize}
In both cases the $W_{p,q,n}$-pair of $c$ is the 2-cell $\{[k-1,k],c_2,\ldots,c_n\}$ provided $k$ is minimal possible satisfying the conditions above.

\smallskip
As reviewed in Subsection \ref{dmtatofg}, generators in the $W_{p,q,n}$-based Morse-Poincar\'e presentation for $B_n(p\times q)$ are given by critical 1-cells, for which we introduce a manageable notation in Definition \ref{notacion} below. Then, using Lemma \ref{flowdynamics} (which is a direct consequence of the construction of $W_{p,q,n}$), we will spell out the relations among the generators. Relations will be classified into four types of reduced forms of boundaries $\partial(c)$ of critical 2-cells $c$.
 
\begin{definition}\label{notacion}
\begin{itemize}
\item For an edge $e=[i,j]$, $i<j$, and non-negative integers $r,s,t$ satisfying the conditions
$$
0\leq r<i,\quad0\leq s<j-i,\quad0\leq t\leq pq-j\quad \text{and}\quad r+s+t=n-1
$$
we use the generic notation $e(r,s,t)$ for any 1-cell of $\UXnpq$ whose edge ingredient is $e$ and has $r$ vertex ingredients $\ell$ with $\ell<i$, $s$ vertex ingredients $\ell$ with $i<\ell<j$, and $t$ vertex ingredients $\ell$ with $j<\ell$.
\item For a deleted edge $e_i$, the unique $W_{p,q,n}$-critical cell of the form $e_i(r,s,t)$ will be denoted by $\varepsilon_i(r,s,t)$.
\end{itemize}
\end{definition}

\begin{lemma}\label{flowdynamics}\hspace{0em}
\begin{enumerate}
\item If the edge ingredient of a 1-cell $c$ is non-deleted, then $c$ is either $W_{p,q,n}$-collapsible or $W_{p,q,n}$-redundant. In the latter case, for any path $c\nearrow W(c)\searrow d$, the edge ingredient of $d$ is also non-deleted. Thus, no upper gradient path $c\nearrow W(c)\searrow d\nearrow\cdots\searrow e$ can end at a $W_{p,q,n}$-critical 1-cell $e$.
\item Assume that the edge ingredient of a 1-dimensional cell $c$ is deleted. Then $c$ is either $W_{p,q,n}$-critical or $W_{p,q,n}$-redundant. Furthermore there is a unique $W_{p,q,n}$-gradient path $c\nearrow W(c)\searrow d\nearrow\cdots\searrow e$ (of length~0, when $c$ is $W_{p,q,n}$-critical) ending at some $W_{p,q,n}$-critical 1-cell $e$. In this $W_{p,q,n}$-gradient path, all non-blocked vertices of $c$ \emph{fall}, one by one starting from the smallest one, through the maximal tree $T_{p,q}$ towards its root util reaching a position blocked in $c$.
\end{enumerate}
\end{lemma}

\begin{corollary}\label{reducidas}
Any 1-cell $c$ of $\UXnpq$ having non-deleted edge ingredient has trivial reduced expression. On the other hand, the reduced expression of a generic cell $e_i(r,s,t)$ is $\varepsilon_i(r,s,t)$.
\end{corollary}
\begin{proof}
The first assertion follows from item 1 in Lemma \ref{flowdynamics}. Likewise, the second assertion is a direct consequence of item 2 in the lemma, although orientation issues need to be carefully accounted this time. Namely, consider a $W_{p,q,n}$-redundant 1-cell $e_i(r,s,t)$ with smallest unblocked vertex $\ell$. The situation with $\ell<i$ is depicted on the left hand-side of Figure \ref{flow}, but the argument below works just as well in general. 

\begin{figure}[h!]
\centering
\begin{tikzpicture}[x=.6cm,y=.6cm]
\draw(1,0)--(9,0)--(9,2)--(1,2);
\node at (.5,0) {$\cdots$};\node at (.5,2) {$\cdots$};
\node at (3,0) {$\bullet$};\node at (5.5,0) {$\bullet$};
\node at (5.5,2) {$\bullet$};
\node at (5.5,2.5) {\footnotesize$j$};
\node at (3,-.4) {\footnotesize$\ell$};
\node at (5.5,-.4) {\footnotesize$i$};
\draw[thick,dotted] (5.5,0)--(5.5,2);

\draw (14,0)--(16,0)--(16,2)--(14,2)--(14,0);
\draw[very thin,->] (15.8,-1.1)--(12.9,-1.1)--(12.9,3)--(15.8,3);
\node at (15.9,-.3) {\footnotesize$\ell$};
\node at (14.4,-.3) {\footnotesize$\ell{-}1$};
\draw[thick,dotted,->] (15,-.3)--(15.57,-.3);
\draw[thick,dotted,->] (14.3,2.3)--(15.7,2.3);
\node at (13.75,.2) {\footnotesize$i$};
\node at (13.75,1.8) {\footnotesize$j$};
\draw[thick,dotted,->] (13.75,.65)--(13.75,1.35);
\draw[thick,dotted,->] (16.25,.27)--(16.25,1.65);
\node at (17.5,1) {\footnotesize$e_i(r,s,t)$};
\draw[dashed,->] (15.7,1)--(14.3,1);
\draw[dashed,->] (15.7,1.4)--(14.8,1.7);
\draw[dashed,->] (15.7,.6)--(14.8,.3);
\end{tikzpicture}
\caption{A typical step in the construction of reduced forms}
\label{flow}
\end{figure}

Ignoring zero dimensional ingredients (vertices), the square on the right hand-side of Figure~\ref{flow} represents the 2-dimensional $W_{p,q,n}$-pair of $e_i(r,s,t)$. Here dashed arrows indicate the associated simple collapse, and dotted arrows indicate the selected orientations for the relevant edges. In this terms, the right hand-side edge of the square stands for our redundant cell $e_i(r,s,t)$, which is then replaced in the reduction process of Subsection \ref{dmtatofg} by the concatenation of (suitable $\pm1$-powers of) the three paths accounted for by the external light arrow. The top and bottom edges of the square can be neglected in view of the first sentence of this corollary. Therefore our original 1-cell $e_i(r,s,t)$ is replaced by a new version of $e_i(r,s,t)$ (with exponent 1, and not $-1$) in which the original smallest unblocked vertex $\ell$ has ``fallen down'' one slot. The result follows after iterating the process indicated by the $W_{p,q,n}$-gradient path in item 2 of Lemma \ref{flowdynamics} until reaching its $W_{p,q,n}$-critical end.
\end{proof}

Recall there are two types of $W_{p,q,n}$-critical 2-cells in $\UXnpq$. The ones coming from $\UF(\Gamma_{p,q},n)$, which are sets consisting of two deleted edges and $n-2$ blocked vertices, and the ones outside $\UDF(\Gamma_{p,q},n)$, which are sets consisting of a square and $n-1$ (non-necessarily blocked) vertices. We describe next the reduced form of the corresponding boundaries.

\medskip\noindent {\bf Case 1}.
Let $c$ be a $W_{p,q,n}$-critical 2-cell coming from $\UF(\Gamma_{p,q},n)$ with deleted edge ingredients $e_i=[i,j]$ and $e_{i'}=[i',j']$, say with $i<i'$. Assume in addition $j'<j$. As indicated on the left hand-side of Figure \ref{caso1}, let $r$ be the number of vertex ingredients $\ell$ in $c$ with $\ell<i$, $s$ the number of vertex ingredients $\ell$ in $c$ with $i<\ell<i'$, $t$ the number of vertex ingredients $\ell$ in $c$ with $i'<\ell<j'$, $u$ the number of vertex ingredients $\ell$ in $c$ with $j'<\ell<j$, and $v$ the number of vertex ingredients $\ell$ in $c$ with $j<\ell$. Thus $r+s+t+u+v=n-2$. As before, the square on the right hand-side of Figure~\ref{caso1} represents $c$, with straight arrows indicating orientations of the four faces of the square. The boundary $\partial(c)$ is given by the concatenation of the four faces, say by following the inner cycle. By Corollary \ref{reducidas}, the resulting relation in $B_n(p\times q)$ is
\begin{equation}\label{relation1}
\varepsilon_i(r,s+t+u+1,v) \cdot
\varepsilon_{i'}(r+s,t,u+v+1) \cdot
\overline{\varepsilon_i(r,s+t+u+1,v)} \cdot
\overline{\varepsilon_{i'}(r+s+1,t,u+v)}.
\end{equation}
Here and throughout the paper we write $\overline{x}$ for the inverse of $x$.

\begin{figure}[h!]
\centering
\begin{tikzpicture}[x=.6cm,y=.6cm]
\draw(1,0)--(9,0)--(9,2)--(1,2);
\node at (.5,0) {$\cdots$};\node at (.5,2) {$\cdots$};
\node at (3,0) {$\bullet$};\node at (5.5,0) {$\bullet$};
\node at (5.5,2) {$\bullet$};\node at (3,2) {$\bullet$};
\node at (5.5,2.5) {\footnotesize$j'$};
\node at (3,-.4) {\footnotesize$i$};
\node at (3,2.5) {\footnotesize$j$};
\node at (5.5,-.4) {\footnotesize$i'$};
\draw[thick,dotted] (5.5,0)--(5.5,2);
\draw[thick,dotted] (3,0)--(3,2);
\node at (4.2,.4) {\footnotesize$s$};
\draw[->] (4.5,.4)--(5.1,.4);\draw[->] (3.9,.4)--(3.25,.4);
\node at (1.8,.4) {\footnotesize$r$};
\draw[->] (2.1,.4)--(2.6,.4);\draw[->] (1.45,.4)--(.8,.4);
\node at (1.8,1.6) {\footnotesize$v$};
\draw[->] (2.1,1.6)--(2.6,1.6);\draw[->] (1.45,1.6)--(.8,1.6);
\node at (4.2,1.6) {\footnotesize$u$};
\draw[->] (4.5,1.6)--(5.1,1.6);\draw[->] (3.9,1.6)--(3.25,1.6);
\node at (8,.95) {\footnotesize$t$};
\draw (6.2,.3) arc(270:450:1.5cm and .4cm);
\draw[->] (6.3,1.633)--(5.9,1.633);
\draw[->] (6.3,.3)--(5.9,.3);

\draw (18.66-1,.4) arc(240:555:.4cm and .4cm);
\draw[->] (18.344-1,.82)--(18.4-1,.68);
\draw (18-1,0)--(20-1,0)--(20-1,2)--(18-1,2)--(18-1,0);
\node at (19.9-1,-.29) {\footnotesize$j$};
\node at (18.1-1,-.3) {\footnotesize$i$};
\draw[->] (4+14.5-1,-.3)--(4+15.5-1,-.3);
\draw[->] (4+14.5-1,2.3)--(4+15.5-1,2.3);
\node at (4+13.75-1,.2) {\footnotesize$i'$};
\node at (4+13.75-1,1.8) {\footnotesize$j'$};
\draw[->] (4+13.65-1,.6)--(4+13.65-1,1.3);
\draw[->] (4+16.3-1,.6)--(4+16.3-1,1.3);
\node at (4+19.15-1,.9) {\footnotesize$e_{i'}(r+s,t,u+v+1)$};
\node at (14.8-1,.9) {\footnotesize$e_{i'}(r+s+1,t,u+v)$};
\node at (19-1,2.76) {\footnotesize$e_{i}(r,s+t+u+1,v)$};
\node at (19-1,-.87) {\footnotesize$e_{i}(r,s+t+u+1,v)$};
\end{tikzpicture}
\caption{$W_{p,q,n}$-critical 2-cell in Case 1 (left) and the resulting relation (right)}
\label{caso1}
\end{figure}

\begin{figure}
\centering
\begin{tikzpicture}[x=.6cm,y=.6cm]
\draw(1,0)--(9,0)--(9,2)--(0,2)--(0,4)--(8,4);
\node at (.5,0) {$\cdots$};\node at (8.7,4) {$\cdots$};
\node at (3,0) {$\bullet$};\node at (5.5,4) {$\bullet$};
\node at (5.5,2) {$\bullet$};\node at (3,2) {$\bullet$};
\node at (5.85,2.5) {\footnotesize$i'$};
\node at (3,-.4) {\footnotesize$i$};
\node at (2.7,1.56) {\footnotesize$j$};
\node at (5.6,4.55) {\footnotesize$j'$};
\draw[thick,dotted] (3,0)--(3,2);\draw[thick,dotted] (5.5,2)--(5.5,4);
\draw[->] (6.22,.3)--(3.25,.3);
\node at (1.8,.4) {\footnotesize$r$};
\draw[->] (2.1,.4)--(2.6,.4);
\draw[->] (1.45,.4)--(.8,.4);
\node at (4.2,1.6) {\footnotesize$t$};
\draw[->] (4.5,1.6)--(5.1,1.6);\draw[->] (3.9,1.6)--(3.25,1.6);
\node at (8,.95) {\footnotesize$s$};
\draw (6.2,.3) arc(270:450:1.5cm and .4cm);
\draw[->] (6.3,1.633)--(5.9,1.633);
\draw (2.8,2.354) arc(270:90:1.5cm and .4cm);
\draw[->] (2.8,3.687)--(5.3,3.687);
\draw[->] (2.8,2.356)--(2.81,2.356);
\node at (1.2,3) {\footnotesize$u$};
\draw[->] (6.4,3.687)--(5.7,3.687);
\node at (6.75,3.687) {\footnotesize$v$};
\draw[->] (7,3.687)--(7.8,3.687);

\draw (18.66-1,.4+1) arc(240:555:.4cm and .4cm);
\draw[->] (18.344-1,.82+1)--(18.4-1,.68+1);
\draw (18-1,0+1)--(20-1,0+1)--(20-1,2+1)--(18-1,2+1)--(18-1,0+1);
\node at (19.9-1,-.29+1) {\footnotesize$j$};
\node at (18.1-1,-.3+1) {\footnotesize$i$};
\draw[->] (4+14.5-1,-.3+1)--(4+15.5-1,-.3+1);
\draw[->] (4+14.5-1,2.3+1)--(4+15.5-1,2.3+1);
\node at (4+13.75-1,.2+1) {\footnotesize$i'$};
\node at (4+13.75-1,1.8+1) {\footnotesize$j'$};
\draw[->] (4+13.65-1,.6+1)--(4+13.65-1,2.3);
\draw[->] (4+16.3-1,.6+1)--(4+16.3-1,1.3+1);
\node at (4+19.15-1,.9+1) {\footnotesize$e_{i'}(r+s,t+u+1,v)$};
\node at (14.8-1,.9+1) {\footnotesize$e_{i'}(r+s+1,t+u,v)$};
\node at (19-1,2.76+1) {\footnotesize$e_{i}(r,s+t,u+v+1)$};
\node at (19-1,-.87+1) {\footnotesize$e_{i}(r,s+t+1,u+v)$};
\end{tikzpicture}
\caption{$W_{p,q,n}$-critical 2-cell in Case 2 (left) and the resulting relation (right)}
\label{caso2}
\end{figure}

\medskip\noindent{\bf Case 2}. Let $c$ be a $W_{p,q,n}$-critical 2-cell coming from $\UF(\Gamma_{p,q},n)$ with deleted edge ingredients $e_i=[i,j]$ and $e_{i'}=[i',j']$, say with $i<i'$. Assume now $j'>j>i'$ and consider the corresponding five non-negative integers $r,s,t,u,v$ depicted on the left hand-side of Figure \ref{caso2} that add up to $n-2$ and describe the distribution of blocked vertices in $c$. This time the resulting relation in $B_n(p\times q)$ is
\begin{equation}\label{relation2}
\varepsilon_i(r,s+t+1,u+v) \cdot
\varepsilon_{i'}(r+s,t+u+1,v) \cdot
\overline{\varepsilon_i(r,s+t,u+v+1)} \cdot
\overline{\varepsilon_{i'}(r+s+1,t+u,v)}.
\end{equation}

\noindent{\bf Case 3}. Let $c$ be a $W_{p,q,n}$-critical 2-cell coming from $\UF(\Gamma_{p,q},n)$ with deleted edge ingredients $e_i=[i,j]$ and $e_{i'}=[i',j']$, say with $i<i'$. The remaining situation for this type of cells has $j<i'$. Consider the corresponding five non-negative integers $r,s,t,u,v$ depicted on the left hand-side of Figure \ref{caso3} that add up to $n-2$ and describe the distribution of blocked vertices in $c$. The resulting relation in $B_n(p\times q)$ is now
\begin{equation}\label{relation3}
\varepsilon_i(r,s,t+u+v+1) \cdot
\varepsilon_{i'}(r+s+t+1,u,v) \cdot
\overline{\varepsilon_i(r,s,t+u+v+1)} \cdot
\overline{\varepsilon_{i'}(r+s+t+1,u,v)}.
\end{equation}

\begin{figure}[h!]
\centering
\begin{tikzpicture}[x=.6cm,y=.6cm]
\draw(1,0)--(9,0)--(9,2)--(0,2)--(0,4)--(8,4);
\node at (.5,0) {$\cdots$};\node at (8.7,4) {$\cdots$};
\node at (5.5,0) {$\bullet$};\node at (3,4) {$\bullet$};
\node at (5.5,2) {$\bullet$};\node at (3,2) {$\bullet$};
\node at (3.4,2.5) {\footnotesize$i'$};
\node at (5.5,-.4) {\footnotesize$i$};
\node at (5.5,2.5) {\footnotesize$j$};
\node at (3.1,4.55) {\footnotesize$j'$};
\draw[thick,dotted] (3,2)--(3,4);\draw[thick,dotted] (5.5,0)--(5.5,2);
\node at (1.8,.3) {\footnotesize$r$};
\draw[->] (2.1,.3)--(5.1,.3);
\draw[->] (1.45,.3)--(.8,.3);
\node at (4.2,1.6) {\footnotesize$t$};
\draw[->] (4.5,1.6)--(5.1,1.6);\draw[->] (3.9,1.6)--(3.25,1.6);
\node at (8,.95) {\footnotesize$s$};
\draw (6.2,.3) arc(270:450:1.5cm and .4cm);
\draw[->] (6.3,1.633)--(5.9,1.633);
\draw[->] (6.3,.3)--(5.9,.3);
\draw (2.8,2.354) arc(270:90:1.5cm and .4cm);
\draw[->] (6.35,3.687)--(3.2,3.687);
\draw[->] (2.8,2.356)--(2.81,2.356);
\draw[->] (2.8,3.686)--(2.81,3.686);
\node at (1.2,3) {\footnotesize$u$};
\node at (6.7,3.687) {\footnotesize$v$};
\draw[->] (7,3.687)--(7.8,3.687);

\draw (18.66-1,.4+1) arc(240:555:.4cm and .4cm);
\draw[->] (18.344-1,.82+1)--(18.4-1,.68+1);
\draw (18-1,0+1)--(20-1,0+1)--(20-1,2+1)--(18-1,2+1)--(18-1,0+1);
\node at (19.9-1,-.29+1) {\footnotesize$j$};
\node at (18.1-1,-.3+1) {\footnotesize$i$};
\draw[->] (4+14.5-1,-.3+1)--(4+15.5-1,-.3+1);
\draw[->] (4+14.5-1,2.3+1)--(4+15.5-1,2.3+1);
\node at (4+13.75-1,.2+1) {\footnotesize$i'$};
\node at (4+13.75-1,1.8+1) {\footnotesize$j'$};
\draw[->] (4+13.65-1,.6+1)--(4+13.65-1,1.3+1);
\draw[->] (4+16.3-1,.6+1)--(4+16.3-1,1.3+1);
\node at (4+19.15-1,.9+1) {\footnotesize$e_{i'}(r+s+t+1,u,v)$};
\node at (14.8-1,.9+1) {\footnotesize$e_{i'}(r+s+t+1,u,v)$};
\node at (19-1,2.76+1) {\footnotesize$e_{i}(r,s,t+u+v+1)$};
\node at (19-1,-.87+1) {\footnotesize$e_{i}(r,s,t+u+v+1)$};
\end{tikzpicture}
\caption{$W_{p,q,n}$-critical 2-cell in Case 3 (left) and the resulting relation (right)}
\label{caso3}
\end{figure}

\noindent{\bf Case 4}. Let $c=\{c_1,\ldots,c_n\}$ be a $W_{p,q,n}$-critical 2-cell outside $\UF(\Gamma_{p,q},n)$, say with square ingredient $c_1$ determined by the four vertices $i,i+1,j-1,j$. This time we only need the three non-negative integers $r,s,t$ depicted in Figure \ref{caso4} that add up to $n-1$ in order to indicate the distribution of (not necessarily blocked) vertices in $c$. The resulting relation in $B_n(p\times q)$ is
\begin{equation}\label{relation4}\begin{gathered}
\varepsilon_i(r,s,t) \cdot
\overline{\varepsilon_{i+1}(r,s,t)}, \text{ when $i+1\not\equiv 0 \text{ mod } p$,} \\
\hspace{2.1cm}\varepsilon_i(r,s,t), \text{ when $i+1\equiv 0 \text{ mod } p$.}
\end{gathered}\end{equation}
Note that the second instance of (\ref{relation4}) holds only with $s=0$.

\begin{figure}[h!]
\centering
\begin{tikzpicture}[x=.6cm,y=.6cm]
\draw(1,0)--(9,0)--(9,2)--(1,2);
\node at (.5,0) {$\cdots$};\node at (.5,2) {$\cdots$};
\node at (5.5,0) {$\bullet$};\node at (3,0) {$\bullet$};
\node at (5.5,2) {$\bullet$};\node at (3,2) {$\bullet$};
\node at (3,2.5) {\footnotesize$j$};
\node at (5.5,-.4) {\footnotesize$i+1$};
\node at (5.5,2.5) {\footnotesize$j-1$};
\node at (3,-.4) {\footnotesize$i$};
\draw[thick,dotted] (3,0)--(3,2);\draw[thick,dotted] (5.5,0)--(5.5,2);
\node at (1.8,.3) {\footnotesize$r$};
\draw[->] (2.1,.3)--(2.7,.3);
\draw[->] (1.45,.3)--(.8,.3);
\node at (8,.95) {\footnotesize$s$};
\draw (6.2,.3) arc(270:450:1.5cm and .4cm);
\draw[->] (6.3,1.633)--(5.9,1.633);
\draw[->] (6.3,.3)--(5.9,.3);
\node at (1.8,1.6) {\footnotesize$t$};
\draw[->] (2.1,1.6)--(2.7,1.6);
\draw[->] (1.45,1.6)--(.8,1.6);
\node at (4.25,.9) {$c_1$};
\draw[very thin] (3.25,0.2)--(4-.125,.7);
\draw[very thin] (4.5,1.2)--(5.25,1.8);
\draw[very thin] (3.25,0.8)--(4.5,1.8);
\draw[very thin] (4.05,0.2)--(5.3,1.2);
\draw[very thin] (3.25,1.37)--(3.75,1.77);
\draw[very thin] (4.75,0.2)--(5.25,0.6);
\end{tikzpicture}
\caption{$W_{p,q,n}$-critical 2-cell outside $\UF(\Gamma_{p,q},n)$ in Case 4}
\label{caso4}
\end{figure}

Relations of the first type above are the only ones relevant for Theorem \ref{2huecos}.

\begin{proof}[Proof of Theorem \ref{2huecos}.] 
Proposition \ref{dualidad} suggests the following efficient notation for generators of $B_{2p-2}(p\times2)$. For an edge $e=[i,j]$, $i<j$, and  a vertex $h$ different from $i$ and $j$, let $e[h]$ stand for the generator coming from the 1-cell of $\UX(2p-2,p\times2)$ consisting of $e$ and all vertices of $\Gamma_{p,2}$ other than $i, j,h$. In these terms, the Morse-Poincar\'e presentation of $B_{2p-2}(p\times2)$ has $3p-5$ generators, namely, $b_1:=e_1[2p-1]$, $a_i:=e_i[i-1]$, $b_i:=e_i[2p-i]$, and $c_i:=e_i[2p]$ for $i\in\{2,3,\ldots,p-1\}$. SSee the left hand-side of Figure \ref{dospemenos2}.
\begin{figure}
\centering
\begin{tikzpicture}[x=.6cm,y=.6cm]
\draw(1,0)--(7,0);\draw(1,2)--(7,2);\node at (7.55,0) {$\cdots$};
\draw(8,0)--(12,0)--(12,2)--(8,2);\node at (7.55,2) {$\cdots$};
\node at (1,0) {$\bullet$};\node at (3,0) {$\bullet$};
\node at (5,0) {$\bullet$};\node at (10,0) {$\bullet$};
\node at (12,0) {$\bullet$};
\node at (1,2) {$\bullet$};\node at (3,2) {$\bullet$};
\node at (5,2) {$\bullet$};\node at (10,2) {$\bullet$};
\node at (12,2) {$\bullet$};
\node at (1,-.4) {\footnotesize$1$};
\node at (3,-.4) {\footnotesize$2$};
\node at (5,-.4) {\footnotesize$3$};
\node at (10,-.4) {\footnotesize$p-1$};
\node at (12,-.4) {\footnotesize$p$};

\node at (1,2.4) {\footnotesize$2p$};
\node at (3,2.4) {\footnotesize$2p-1$};
\node at (5,2.4) {\footnotesize$2p-2$};
\node at (10,2.4) {\footnotesize$p+2$};
\node at (12,2.4) {\footnotesize$p+1$};
\draw[thick,dotted] (1,0)--(1,2);\draw[thick,dotted] (3,0)--(3,2);
\draw[thick,dotted] (5,0)--(5,2);\draw[thick,dotted] (10,0)--(10,2);
\node at (.6,.95) {\footnotesize$e_1$};
\node at (2.6,.95) {\footnotesize$e_2$};
\node at (4.6,.95) {\footnotesize$e_3$};
\node at (9.34,.95) {\footnotesize$e_{p-1}$};

\draw (18.66,.4) arc(240:555:.4cm and .4cm);
\draw[->] (18.344,.82)--(18.4,.68);
\draw (18,0)--(20,0)--(20,2)--(18,2)--(18,0);
\node at (19.9,-.29) {\footnotesize$j$};
\node at (18.1,-.3) {\footnotesize$i$};
\draw[->] (4+14.5,-.3)--(4+15.5,-.3);
\draw[->] (4+14.5,2.3)--(4+15.5,2.3);
\node at (4+13.75,.2) {\footnotesize$i'$};
\node at (4+13.75,1.8) {\footnotesize$j'$};
\draw[->] (4+13.65,.6)--(4+13.65,1.3);
\draw[->] (4+16.3,.6)--(4+16.3,1.3);
\node at (21.2,.9) {\footnotesize$e_{i'}[i]$};
\node at (16.6,.9) {\footnotesize$e_{i'}[j]$};
\node at (19,2.84) {\footnotesize$e_{i}[i']$};
\node at (19,-.87) {\footnotesize$e_{i}[j']$};
\end{tikzpicture}
\caption{Generators and relations in $B_{2p-2}(p\times 2)$}
\label{dospemenos2}
\end{figure}
Furthermore, there is room only for the relations in Case 1 above. Actually, for each pair of deleted edges $e_i=[i,j]$ and $e_{i'}=[i',j']$ with $i<i'$, there is a single $W_{p,2,2p-2}$-critical 2-cell with edges $e_i$ and $e_{i'}$ and, as illustrated on the right hand-side of Figure \ref{dospemenos2}, the corresponding relation (\ref{relation1}) becomes
\begin{equation}\label{getridof}
e_i[j'] \cdot e_{i'}[i] \cdot \overline{e_i[i']} \cdot \overline{e_{i'}[j]}= b_i \hspace{.3mm} a_{i'} \hspace{.3mm} \overline{b_i} \hspace{.3mm} \overline{c_{i'}},
\end{equation}
in view of Corollary \ref{reducidas}. The resulting presentation can then be simplified by using (\ref{getridof}) to get rid of $c_{i'}=b_1 a_{i'} \overline{b_1}$ for $1<i'\leq p-1$. This yields the presentation with generators $b_1$, $a_i$ and $b_i$ for $2\leq i\leq p-1$, subject to the relations $b_i a_{i'} \overline{b_i} b_1 \overline{a_{i'}} \overline{b_1}$ for $2\leq i<i'\leq p-1$. Then the rules 
$$
\begin{tabular}{lcr}
$a_i\mapsto x_i$ & 
\hspace{.5cm}for $2\leq i\leq p-1\hspace{.5cm}$ & 
$a_i \mapsfrom x_i$ \\
$b_i \mapsto y_1\overline{y_i}$ & 
\hspace{.5cm}for $2\leq i\leq p-1\hspace{.5cm}$ &
$\overline{b_i}b_1 \mapsfrom y_i$\\
$b_1\mapsto y_1$ & &
$b_1 \mapsfrom y_1$
\end{tabular}
$$
define an isomorphism between $B_{2p-2}(p\times2)$ and the group with generators $x_2,x_3,\ldots,x_{p-1}$ and $y_1,y_2,y_3\ldots,y_{p-1}$ subject to relations
$y_1 \overline{y_i} \hspace{.4mm}x_j y_i \hspace{.4mm}\overline{x_j} \hspace{.5mm}\overline{y_1}$
for $2\leq i<j\leq p-1$. The proof is complete since the latter relation simply says that $y_i$ and $x_j$ commute. 
\end{proof}

In general, the presentation of $B_n(p\times q)$ in terms of generators $\varepsilon_i(r,s,t)$ and their relations (\ref{relation1})--(\ref{relation4}) discussed in this section is far from optimal. For $q=2$, in the final two sections of the paper we first give a major reduction of generators, and then a corresponding simplification of relations. This will allow us to identify the right-angled Artin group structure asserted in Theorem \ref{combgen} for $B_n(p\times 2)$ when $p\leq n\leq 2p-5$.

\section{\texorpdfstring{Minimal set of generators and key relations in $B_n(p\times2)$}{Minimal generators and key relations}}
The rest of the paper is devoted to the case $q=2$.

\smallskip
In this section we work under the preliminary assumption $p\leq n\leq 2p-3$. In particular $3\leq p\leq n$. (The stronger assumption $p\leq n\leq 2p-5$ will be required only in Proposition~\ref{otrasrels} near the end of the section.) Under these conditions there is room only for the relations in Cases 1 and 4 of Section \ref{rawpresentation}. This section's goal is to spell out a reduced form of the resulting presentation.

\smallskip
Following the conventions set forth in the previous section, for $1\leq i\leq p-1$ and $r,s,t\geq0$ with $r+s+t=n-1$, we have a generator $\varepsilon_i(r,s,t)$ provided it is possible to fit $r$, $s$ and $t$ blocked vertices so to assemble the critical 1-cell with configuration
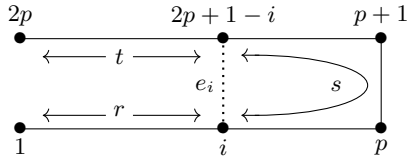
\begin{figure}[h!]
\centering
\begin{tikzpicture}[x=.6cm,y=.6cm]
\draw(1,0)--(9,0)--(9,2)--(1,2);
\node at (1,0) {$\bullet$};\node at (9,0) {$\bullet$};
\node at (1,2) {$\bullet$};\node at (9,2) {$\bullet$};
\node at (5.5,0) {$\bullet$};\node at (5.5,2) {$\bullet$};
\node at (5.5,2.5) {\footnotesize$2p+1-i$};
\node at (5.5,-.4) {\footnotesize$i$};
\node at (1,-.4) {\footnotesize$1$};
\node at (9,-.4) {\footnotesize$p$};
\node at (1,2.5) {\footnotesize$2p$};
\node at (9,2.5) {\footnotesize$p+1$};
\draw[thick,dotted] (5.5,0)--(5.5,2);
\node at (3.2,.4) {\footnotesize$r$};
\draw[->] (3.5,.3)--(5,.3);\draw[->] (2.9,.3)--(1.5,.3);
\node at (3.2,1.6) {\footnotesize$t$};
\draw[->] (3.5,1.6)--(5,1.6);\draw[->] (2.9,1.6)--(1.5,1.6);
\node at (5.1,.95) {\footnotesize$e_i$};
\node at (8,.95) {\footnotesize$s$};
\draw (6.2,.3) arc(270:450:1.5cm and .4cm);
\draw[->] (6.3,1.633)--(5.9,1.633);
\draw[->] (6.3,.3)--(5.9,.3);
\end{tikzpicture}
\end{figure}

\noindent The conditions for the existence of the generator $\varepsilon_i(r,s,t)$ can then be summarized as
\begin{equation}\label{genconditions}
1\leq i\leq p-1,\quad r+s+t=n-1,\quad 0\leq r,t\leq i-1\quad\text{and}\quad 0\leq s\leq 2(p-i)
\end{equation}
(cf.~Definition \ref{notacion}). Likewise, the left hand-side in Figure \ref{caso1} makes it clear that a relation (\ref{relation1}) holds whenever
\begin{equation}\label{relconditions}\begin{gathered}
1\leq i<i'\leq p-1,\quad r+s+t+u+v=n-2, \\
0\leq r,v\leq i-1,\quad 0\leq s,u\leq i'-i-1\quad \text{and}\quad 
0\leq t\leq 2(p-i').
\end{gathered}\end{equation}
Relations (\ref{relation4}) allow us to make a first substantial reduction of generators:

\begin{lemma}\label{generatorsreduction}
Let $i<i'$. If $\varepsilon_i(r,s,t)$ and $\varepsilon_{i'}(r,s,t)$ are allowed generators (i.e., each satisfying the corresponding conditions (\ref{genconditions})), then so is $\varepsilon_k(r,s,t)$ for $i\leq k\leq i'$. Furthermore all of these generators agree in $B_n(p\times2)$, and the resulting element will simply be denoted by $\varepsilon(r,s,t)$. For $s=0$, $\varepsilon(r,s,t)$ is the trivial element.
\end{lemma}
\begin{proof}
The verification of the first assertion is elementary. The other two assertions follow directly from relations~(\ref{relation4}).
\end{proof}

In view of the third condition in (\ref{genconditions}), the first integer $i\in\{1,2,\ldots,p-1\}$ for which $\varepsilon_i(r,s,t)$ would make sense is $\iota(r,t):=\max(r,t)+1$. So, taking into account Lemma \ref{generatorsreduction}, conditions (\ref{genconditions}) become 
\begin{equation}\label{genconds}
r,t\geq0<s,\quad r+s+t=n-1,\quad \iota(r,t)\leq p-1\quad \text{and}\quad s\leq2(p-\iota(r,t)),
\end{equation}
which characterize the existence of a ``unified generator'' $\varepsilon(r,s,t)$ for $B_n(p\times2)$. Likewise, conditions (\ref{relconditions}) become
\begin{equation}\label{relconds}\begin{gathered}
r,s,u,v\geq0<t,\quad r+s+t+u+v=n-2, \\
\iota(r,v)+\iota(s,u)\leq p-1,\quad \text{and}\quad
t\leq 2(p-\iota(r,v)-\iota(s,u)),
\end{gathered}\end{equation}
which characterize the existence of a ``unified relation''
\begin{equation}\label{rel1}
\varepsilon(r,s+t+u+1,v) \cdot
\varepsilon(r+s,t,u+v+1) \cdot
\overline{\varepsilon(r,s+t+u+1,v)} \cdot
\overline{\varepsilon(r+s+1,t,u+v)}.
\end{equation}
The main task in this section is to make an additional simplification of the resulting ``unified presentation''. 

\medskip
Set $\varphi:=n-p$, so $0\leq\varphi\leq p-3$ and $2\varphi\leq n-3$. Then, for $1\leq j\leq n-1$, set $m_j:=\max(0,\varphi-\lfloor \hspace{.6mm}j/2\rfloor)$. For typographical reasons, at times it will be best to use the notation $m(j)$ instead of $m_j$. It is straightforward to check that $g_j:=\varepsilon(m_j,j,n-1-j-m_j)$ satisfies (\ref{genconds}), so this is a honest element of $B_n(p\times2)$.

\begin{proposition}\label{minimalgenerators}
The elements $g_1,\ldots,g_{n-1}$ generate $B_n(p\times2)$. For $n\leq 2p-5$, this is in fact a minimal set of generators.
\end{proposition}

Minimality in Proposition \ref{minimalgenerators} for $n\leq2p-5$ will follow once we identify, in the next section, a right-angled Artin group structure on $B_n(p\times2)$ with $n-1$ generators. The following considerations set the grounds for proving the first assertion in Proposition \ref{minimalgenerators}.

\begin{notation}{\em
For elements $x$ and $y$ of a group, the $x$-conjugation-power of $y$ is $y^x:=xy\overline{x}$ (recall $\overline{x}$ stands for the inverse of $x$). The obvious relations $z^{xy}=(z^y)^x$, $(xy)^z=x^zy^z$ and $\overline{y^x}=\overline{y}^x$ will be freely used throughout the paper.
}\end{notation}

\begin{lemma}\label{genrestriccion}
Any element $\varepsilon(r,s,t)$ with $r\leq\varphi$ satisfies $2(\varphi-r)\leq s$. In particular $r\geq m_s$.
\end{lemma}
\begin{proof}
Assume for a contradiction $0\leq r\leq\varphi$ and $s<2(\varphi-r)$. If $r\leq t$, then the second and fourth conditions in (\ref{genconds}) yield $s\leq2(p-t-1)=2(p-n+r+s)=-2(\varphi-r)+2s$, which is incompatible with the assumption. Likewise, if $t<r$, then $$n-1=r+s+t\leq r+(2\varphi-2r-1)+(r-1)=2\varphi-2=2(n-p)-2,$$ so that $2p+1\leq n$, which is incompatible with this section's hypothesis that $n\leq 2p-3$. 
\end{proof}

A relation (\ref{rel1}) having $r=v=0$ can be written as
\begin{equation}\label{reduccion}
\varepsilon(s+1,t,u)=\varepsilon(s,t,u+1)^{g_{n-1}},
\end{equation}
which holds whenever
\begin{equation}\label{reduccionconditions}
\mbox{$s,u\geq0<t$, \ \ $s+t+u=n-2$, \ \ $\iota(s,u)\leq p-2$ \ \ and \ \ $t\leq2(p-1-\iota(s,u))$.}
\end{equation}
The first assertion in Proposition \ref{minimalgenerators} is then a consequence of the following result, whose proof requires a slightly involved arithmetical argument (included for completeness).

\begin{proposition}\label{iteratedpowers}
For any element $\varepsilon(a,b,c)$ we have $a\geq m_b$. Furthermore, conditions (\ref{reduccionconditions}) hold for $s:=a-1$, $t:=b$ and $u:=c$ whenever $a>m_b$.
\end{proposition}

\begin{proof}
By (\ref{genconds}), our hypotheses are
\begin{equation}\label{updatedgenconds}
a,c\geq0<b,\quad a+b+c=n-1,\quad \iota(a,c)\leq p-1\quad \text{and}\quad b\leq2(p-\iota(a,c)).
\end{equation}
If $a=0$, Lemma \ref{genrestriccion} forces in fact $a=m_b=0$. So we can assume $a>0$, with the first two conditions in (\ref{reduccionconditions}) holding on the nose. In order to check the last two conditions in (\ref{reduccionconditions}), we start by noticing that the third condition in (\ref{updatedgenconds}) gives
\begin{equation}\label{laazul}
\iota(s,u)=\iota(a-1,c)=\begin{cases}
\iota(a,c)=c+1\leq p-1, & \mbox{if $a\leq c$;} \\
\iota(a,c)-1\leq p-2, &  \mbox{otherwise.}
\end{cases}
\end{equation}

\noindent{\bf Case $a>\varphi$}. Note that $a>m_b$, as $\varphi\geq m_j$ for any $j$. The third condition in (\ref{reduccionconditions}) follows from (\ref{laazul}) by observing that, if $c=p-2$, we would get
$$
n-1=a+b+c\geq(\varphi+1)+b+p-2=n+b-1,
$$
which is incompatible with $b>0$ in (\ref{updatedgenconds}). Likewise, the fourth condition in (\ref{reduccionconditions}) follows from (\ref{laazul}) and the fourth condition in (\ref{updatedgenconds}). Namely, if $a\leq c$ but $b\geq2(p-c)-3$, we would get
$$
n-1=a+b+c\geq(\varphi+1)+(2p-2c-3)+c=n+p-c-2,
$$
i.e., $c+1\geq p$, which is incompatible with the top line on the right hand-side of (\ref{laazul}).

\medskip\noindent{\bf Case $0<a\leq\varphi$}. Lemma \ref{genrestriccion} gives $a\geq m_b$. Assume in fact $a>m_b$. Say
\begin{equation}\label{ahora}
\mbox{$a=\varphi-i$ \ with \ $0\leq i<\varphi$.}
\end{equation}
Lemma \ref{genrestriccion} also gives $b\geq2(\varphi-a)=2i$, while the assumption $a>m_b$ yields in fact $b\geq2i+2$. As in the previous case, the third condition in (\ref{reduccionconditions}) follows from (\ref{laazul}) by observing that, if $c=p-2$, we would get
$$
n-1=a+b+c\geq(\varphi-i)+(2i+2)+p-2=n+i,
$$
which is incompatible with (\ref{ahora}). Likewise, the fourth condition in (\ref{reduccionconditions}) follows from (\ref{laazul}) and the fourth condition in (\ref{updatedgenconds}). Explicitly, if $a\leq c$ and $b\geq2(p-c)-3$, we would get
$$
n-1=a+b+c\geq(\varphi-i)+(2p-2c-3)+c=n+p-c-i-3,
$$
so that $c\geq p-i-2$ and then $n-1=a+b+c\geq(\varphi-i)+(2i+2)+(p-i-2)=n$, which is impossible.
\end{proof}

We have shown how the form (\ref{reduccion}) of the relations (\ref{rel1}) having $r=v=0$ gives a systematic way to write any given $\varepsilon(r,s,t)\in B_n(p\times2)$ as an iterated $g_{n-1}$-conjugation-power of one of the ``reduced'' generators $g_1,\ldots,g_{n-1}$, namely
$$
\varepsilon(r,s,t)=g_s^{\left(g_{n-1}^{r-m_s}\right)}=:g_s^{g_{n-1}^{r-m_s}}.
$$
In these terms, the unified relations (\ref{rel1}) take the ``reduced'' form
\begin{equation}\label{unifiedrelation}
g_{s+t+u+1}^{g_{n-1}^{r-m(s+t+u+1)}}
g_{t}^{g_{n-1}^{r+s-m(t)}}
\overline{g_{s+t+u+1}}^{\hspace{.4mm}g_{n-1}^{r-m(s+t+u+1)}}
\overline{g_{t}}^{\hspace{.4mm}g_{n-1}^{r+s+1-m(t)}},
\end{equation}
which holds under a ``reduced'' version of (\ref{relconds}), namely
\begin{equation}\label{unifiedrelconds}\begin{gathered}
r,s,u,v\geq0<t,r+v,\quad r+s+t+u+v=n-2, \\
\iota(r,v)+\iota(s,u)\leq p-1\quad \text{and}\quad
t\leq 2(p-\iota(r,v)-\iota(s,u)).
\end{gathered}\end{equation}
Actually, since there is a one-to-one correspondence between relations (\ref{reduccion}) and elements $\varepsilon(r,s,t)$ different from the reduced generators, we get a ``reduced'' presentation of $B_n(p\times2)$ consisting of reduced generators and reduced relations. We close the section by identifying (in Propositions \ref{relsconab0}, \ref{otrasrels} and \ref{relsfinales} below) a subset of reduced relations which, as proved in the next section, is directly responsible for a right-angled Artin group structure on $B_n(p\times2)$ when $n\leq2p-5$.

\begin{notation}{\em
For elements $x$ and $y$ of a group, we write ``$x\,\&\,y$'' to mean that $x$ and $y$ commute. This convention is motivated by the standard graph representation of a right-angled Artin group. Indeed, an expression of the form $x\,\&\,y$ suggests the existence of an edge between $x$ and $y$.
}\end{notation}

The following fact is straightforward to check.
\begin{lemma}\label{notacionampersant}
In any group, the relation $xy\hspace{.2mm}\overline{x}\hspace{.4mm}\overline{y}^z=1$ is equivalent to $y\,\&\,\overline{x}z$
\end{lemma}

For instance, since a given relation is equivalent to any of it conjugates, we see that the reduced relations (\ref{unifiedrelation}) can be normalized to
$$
g_{s+t+u+1}^{g_{n-1}^{m(t)-m(s+t+u+1)-s}}
g_{t}\hspace{1mm}
\overline{g_{s+t+u+1}}^{\hspace{.4mm}g_{n-1}^{m(t)-m(s+t+u+1)-s}}
\overline{g_{t}}^{\hspace{.4mm}g_{n-1}}
$$
or, in view of Lemma \ref{notacionampersant}, to
\begin{equation}\label{normalizadas}
g_t\;\;\&\;\;\overline{g_{s+t+u+1}}^{\hspace{.4mm}g_{n-1}^{m(t)-m(s+t+u+1)-s}} g_{n-1}.
\end{equation}
Note that parameters $r$ and $v$ play a role in (\ref{unifiedrelconds}), but not in (\ref{normalizadas}).

\begin{proposition}\label{relsconab0}
For $r=s=0$, relations (\ref{normalizadas}) holding under conditions (\ref{unifiedrelconds}) can be expressed as relations $g_i\,\&\, \overline{g_j} g_{n-1}$ holding whenever $\max(1,2\varphi)\leq i<j\leq n-2$.
\end{proposition}
\begin{proof}
For $r=s=0$, conditions (\ref{unifiedrelconds}) become
$$
u\geq0<t,v,\quad t+u+v=n-2,\quad
v+u\leq p-3\quad \text{and}\quad
t\leq 2(p-v-u-2).
$$
The first two of these conditions allow us to write the third one as $\varphi+1\leq t$, and the fourth one as $\max(1,2\varphi)\leq t$, the latter one being in fact stronger than the former one. Thus, conditions (\ref{unifiedrelconds}) simplify to 
$$
u\geq0<v,\quad t+u+v=n-2\quad
\text{and}\quad
\max(1,2\varphi)\leq t.
$$
Since $u$ is determined by the second of these conditions and the values of  $t$ and $v$, conditions (\ref{unifiedrelconds}) can further be simplified to
$$
0<v\leq n-t-2\quad\text{and}\quad
\max(1,2\varphi)\leq t.
$$
Under these conditions $n-v-1>t\geq2\varphi$, so that $m(t)=m(n-v-1)=0$, and (\ref{normalizadas}) then becomes $g_t\,\&\,\overline{g_{n-v-1}}g_{n-1}$. Take $i:=t$ and $j:=n-v-1$ to get the asserted expressions.
\end{proof}

\begin{example}\label{changen=p}{\em
As a warm-up for the general situation, here is an alternative proof, based on the considerations above, of Theorem \ref{combgen} in the case $n=p$: The rules 
$$
\begin{tabular}{lcr}
$g_i\mapsto h_i$ & $i=1,2$ & 
$g_i \mapsfrom h_i$ \\
$g_i \mapsto h_2h_3\cdots h_i$ & 
\hspace{.5cm}for $3\leq i\leq p-1\hspace{.5cm}$ &
$\overline{g_{i-1}}g_i \mapsfrom h_i$
\end{tabular}
$$
define an isomorphism between $B_p(p\times2)$ and the group with generators $h_1,h_2,\ldots,h_{p-1}$ subject to the relations 
\begin{align*}
h_1\;\&\; h_{j+1}h_{j+2}\cdots h_{p-1}, & 
\text{ \ for $1<j\leq p-2$,}\\
h_2h_3\cdots h_i\;\&\; h_{j+1}h_{j+2}\cdots h_{p-1}, & 
\text{ \ for $2\leq i<j\leq p-2$.}
\end{align*}
An easy inductive argument shows that these relations can equivalently be written as $h_i\;\&\;h_j$ for $|i-j|>1$ (see the first inductive argument in the proof of Proposition \ref{bestorganization}). This yields the right-angled Artin group structure on $B_p(p\times 2)$ since, as proved in the next section, all relations (\ref{normalizadas}) can be derived from those having $r=s=0$, when $n=p$. 
}\end{example}

Set $\ell(i):=2\lfloor(i+1)/2\rfloor$. So $\ell(i)=i$ for even $i$, while $\ell(i)=i+1$ for odd $i$. The following two key facts arose after extensive computer calculations.
\begin{proposition}\label{otrasrels}
Assume $n\leq 2p-5$ and fix integers $i$ and $j$ with $1\leq i<2\varphi$ and $\ell(i)<j\leq n-2$ (so $\varphi>0$ and $n>p$). Then $g_i\,\&\,\overline{g_j}g_{n-1}$ is the instance of (\ref{normalizadas}) arising from the case of (\ref{unifiedrelconds}) with parameters $(r,s,t,u,v)$ given by
\begin{align*}
(r,s,&t,u,v)\\
&=\left(\varphi+\ell(i)-i-\left\lfloor\frac{j-1}2\right\rfloor,\left\lfloor\frac{j+\ell(i)}2\right\rfloor-i, i, \left\lfloor\frac{j-\ell(i)-1}2\right\rfloor, p+i-\ell(i)-2-\left\lfloor\frac{j}2\right\rfloor\right)
\end{align*}
provided $\varphi>\left\lfloor\frac{j+i-\ell(i)-1}2\right\rfloor$, and by
$$
(r,s,t,u,v)=\left(0\hspace{.6mm},\,\, \varphi-i+\left\lfloor\frac{i+1}2\right\rfloor,\,\, i\hspace{.6mm},\,\, j-\varphi-\left\lfloor\frac{i+3}2\right\rfloor,\,\, n-j-1\right)
$$
provided $\varphi\leq\left\lfloor\frac{j+i-\ell(i)-1}2\right\rfloor$.
\end{proposition}
\begin{proof}
If $\varphi>\left\lfloor\frac{j+i-\ell(i)-1}2\right\rfloor$ (so $j\leq2\varphi+1$, holding with a strict inequality when $i$ is even), straightforward verification shows
\begin{equation}\label{laqueusalahipotesis}
s\geq u\geq0, \quad t>0, \quad v\geq r>0 \quad \mbox{and} \quad r+s+t+u+v=n-2,
\end{equation}
which yields the conditions in the first line of (\ref{unifiedrelconds}) together with the equalities $\iota(r,v)=v+1$ and $\iota(s,u)=s+1$. The last two conditions in (\ref{unifiedrelconds}) are then verified by direct inspection. Likewise, direct verification shows $t=i$, $s+t+u+1=j$ and $m(i)-m(j)-s=0$, giving the asserted form of (\ref{normalizadas}). Details can easily be completed by the reader, once the following two useful observations are remarked. First, the verification of $m(i)-m(j)-s=0$ uses the observation that $j\leq2\varphi+1$, which holds with strict inequality when $i$ is even. Second, the hypothesis $n\leq 2p-5$ is needed (only) for the verification of $v\geq r$ in (\ref{laqueusalahipotesis}) when $i$ is odd.

Details for $\varphi\leq\left\lfloor\frac{j+i-\ell(i)-1}2\right\rfloor$ are parallel and much simpler. In fact this time there is no need of giving special indications for the reader interested in working out details, other than a suggestion to start by checking the conditions $u\geq s\geq0$, $\,t>0$, $\,v>r=0$ \,and\, $r+s+t+u+v=n-2$.
\end{proof}

Relations in Propositions \ref{relsconab0} and \ref{otrasrels} can be summarized as
\begin{equation}\label{allrels}
\mbox{$g_i\,\&\,\overline{g_{j}}g_{n-1}$ \ for \ $1\leq i<j\leq n-2$,}
\end{equation}
i.e., those identified in Example \ref{changen=p} as giving the right-angled Artin group structure when $n=p$, \emph{except} for the cases where $j=i+1$ with $i$ odd and smaller than $2\varphi$. In such exceptional cases, the missing relation (in a slightly different format) is accounted for by:

\begin{proposition}\label{relsfinales}
For an odd integer $i$ with $1\leq i<2\varphi$, the relation $g_i\,\&\,g_{n-1}\overline{g_{i+1}}$ is the instance of (\ref{normalizadas}) arising from the case of (\ref{unifiedrelconds}) with parameters
$$
(r,s,t,u,v)=\left(\varphi-\left\lfloor\frac{i}2\right\rfloor,\,\, 0\hspace{.6mm},\,\, i\hspace{.6mm},\,\, 0\hspace{.6mm},\,\, p-2-\left\lfloor\frac{i+1}2\right\rfloor\right).
$$ 
\end{proposition}
\begin{proof}
This is again a straightforward verification. Details are more accessible to check by noticing first that $s=u=0$ and $v\geq r>0$. The relation in this case is $g_i\,\&\,\overline{g_{i+1}}^{g_{n-1}}g_{n-1}$, a rewrite of the one asserted. Worth remarking is the fact that the proof does not require the stronger hypothesis $n\leq2p-5$ needed in Proposition \ref{otrasrels} ---this section's general hypothesis $p\leq n\leq 2p-3$ suffices.
\end{proof}

\section{\texorpdfstring{Right-angled Artin group structure on $B_n(p\times2)$}{Proof conclusion}}
Throughout this section we assume $p\leq n\leq2p-5$. The relations in $B_n(p\times2)$ pinpointed by Propositions \ref{relsconab0}, \ref{otrasrels} and \ref{relsfinales} can be summarized as
\begin{align}
g_i\,\&\,\overline{g_j}g_{n-1}, &\text{ \ for $\max(1,2\varphi)\leq i<j\leq n-2$,}\nonumber\\
g_i\,\&\,\overline{g_j}g_{n-1}, &\text{ \ for $1\leq i=2k+\epsilon<2\varphi$ and $i+\epsilon<j\leq n-2$, where $\epsilon\in\{0,1\}$,}\label{codificados}\\
g_i\,\&\,g_{n-1}\overline{g_{i+1}}, &\text{ \ for $1\leq i=2k+1<2\varphi$.}\nonumber
\end{align}
Let $G_n$ stand for the group generated by symbols $g_1,g_2,\ldots,g_{n-1}$ subject exclusively to the three types of relations (\ref{codificados}). In particular, there is a canonical projection
\begin{equation}\label{canoproye}
\pi\colon G_n\to B_n(p\times2).
\end{equation}

In this section we show (see Proposition \ref{cambiodebasefinal}) that relations (\ref{codificados}) encode the right angled Artin group $H_{n}$ generated by elements $h_1,h_2,\ldots,h_{n-1}$ subject to the commutativity relations
\begin{equation}\label{artingenesis}
h_i\,\&\,h_j \text{ \ for \ } |i-j|>1.
\end{equation}
We then prove (see Proposition \ref{encompassing}) that relations (\ref{codificados}) actually encompass all relations (\ref{normalizadas}), thus settling the presentation asserted in Theorem \ref{combgen} for $B_n(p\times2)$ as a right-angled Artin group when $p\leq n\leq2p-5$.

\medskip
Concerning the first goal, and as a result of the deviation in Proposition \ref{relsfinales} from (\ref{allrels}), we start with an adaptation of the change of basis in Example \ref{changen=p}. Namely, the formul\ae\ $\phi(g_1)=h_1$ and, for $i\geq2$,
$$
\phi(g_i)=\begin{cases}
h(2,3)h(4,5)\cdots h(2k-2,2k-1)h_{2k}, & \hspace{-3mm}i=2k\leq2\varphi;\\
h(2,3)h(4,5)\cdots h(2k-2,2k-1)h(2k,2k+1), & \hspace{-3mm}i =2k+1 \leq2\varphi+1;\\
h(2,3)h(4,5)\cdots h(2\varphi-2,2\varphi-1)h(2\varphi,2\varphi+1)h_{2\varphi+2}\cdots h_i, & \hspace{-3mm}2\varphi+2\leq i\leq n-1,
\end{cases}
$$
determine a morphism $\phi\colon F(g_1,\ldots,g_{n-1})\to F(h_1,\ldots,h_{n-1})$, where $F(x_1,\ldots,x_m)$ stands for the free group generated by $x_1,\ldots,x_m$, and $h(a,b):=h_bh_a$ for $1\leq a,b\leq n-1$ (note the swapping of factors with respect to parameters). It is clear that $\phi$ is surjective. Actually $\phi$ is an isomorphism, as free groups of finite rank are Hopfian. The obvious relation
\begin{equation}\label{furthernotice}
\phi(\overline{g_i}g_{n-1})=\begin{cases}
h_{i+1}h_{i+2}\cdots h_{n-1}, & \!\!\!\!\mbox{$i\geq\max(2,2\varphi+1)$;} \\
h(2\ell+2,2\ell+3)\cdots h(2\varphi,2\varphi+1) h_{2\varphi+2} \cdots h_{n-1}, & \!\!\!\!\mbox{$2\leq i=2\ell+1<2\varphi$;}\\
\overline{h_{2\ell}}\cdot h(2\ell,2\ell+1)\cdots h(2\varphi,2\varphi+1) h_{2\varphi+2} \cdots h_{n-1}, & 
\!\!\!\!\mbox{$2\leq i=2\ell\leq2\varphi$}
\end{cases}
\end{equation}
will be needed at the beginning of the proof of Proposition \ref{cambiodebasefinal}.

\begin{proposition}\label{cambiodebasefinal}
For $p\leq n\leq2p-5$, the set of $\phi$-images of relations~(\ref{codificados}) is equivalent to the set of relations (\ref{artingenesis}). Consequently, $\phi$ induces an isomorphism (also denoted by) $\phi\colon G_n\to H_n$. 
\end{proposition}

\begin{proposition}\label{encompassing}
For $p\leq n\leq2p-5$, the projection in (\ref{canoproye}) is in fact an isomorphism.
\end{proposition}

In proving Propositions \ref{cambiodebasefinal} and \ref{encompassing}, it is useful to bare in mind the following elementary facts:
\begin{lemma}\label{propiedadeselementales}
For elements $a$, $b$ and $c$ of a given group,
\begin{itemize}
\item relations $a\,\&\,b$ and $a\,\&\,bc$ are equivalent to relations $a\,\&\,b$ and $a\,\&\,c$;
\item relations $a\,\&\,b$ and $a\,\&\,cb$ are equivalent to relations $a\,\&\,b$ and $a\,\&\,c$;
\item relation $a\,\&\,b$ is equivalent to relation $a\,\&\,\overline{b}$.
\end{itemize}
\end{lemma}
 
\begin{proof}[Proof of Proposition \ref{cambiodebasefinal}]
The case $n=p$, i.e.~$\varphi=0$ has been addressed in Example \ref{changen=p}. Here we consider the situation for $\varphi>0$. Let us start by spelling out the set of $\phi$-images of relations~(\ref{codificados}):
\begin{enumerate}[(i)]
\item\label{AA} The $\phi$-image of $g_i\,\&\,\overline{g_j}g_{n-1}$ with $2\varphi\leq i<j\leq n-2$ is as follows:
\begin{enumerate}
\item\label{AAaa}
For $2\varphi+1\leq i<j\leq n-2$,
$$h(2,3)\cdots h(2\varphi,2\varphi+1)h_{2\varphi+2}\cdots h_i\,\,\&\,\,h_{j+1}\cdots h_{n-1}.$$

\item\label{AAbb}
For $2\varphi=i<j\leq n-2$,
$$h(2,3)\cdots h(2\varphi-2,2\varphi-1)h_{2\varphi}\,\,\&\,\,h_{j+1}\cdots h_{n-1}.$$
\end{enumerate}

\item\label{BB} The $\phi$-image of $g_i\,\&\,\overline{g_j}g_{n-1}$ with $2\leq i=2k<2\varphi$ and $i<j\leq n-2$ is as follows:
\begin{enumerate}
\item\label{BBaa}
For $2\leq i=2k<2\varphi$ and $2\varphi+1\leq j\leq n-2$,
$$h(2,3)\cdots h(2k-2,2k-1)h_{2k}\,\,\&\,\,h_{j+1}\cdots h_{n-1}.$$

\item\label{BBbb}
For $2\leq i=2k<2\varphi$ and $i<j=2\ell\leq2\varphi$,
$$h(2,3)\cdots h(2k-2,2k-1)h_{2k}\,\,\&\,\,\overline{h_{2\ell}}h(2\ell,2\ell+1)\cdots h(2\varphi,2\varphi+1)h_{2\varphi+2}\cdots h_{n-1}.$$

\item\label{BBcc}
For $2\leq i=2k<2\varphi$ and $i<j=2\ell+1<2\varphi$,
$$h(2,3)\cdots h(2k-2,2k-1)h_{2k}\,\,\&\,\,h(2\ell+2,2\ell+3)\cdots h(2\varphi,2\varphi+1)h_{2\varphi+2}\cdots h_{n-1}.$$
\end{enumerate}

\item\label{CC} The $\phi$-image of $g_i\,\&\,\overline{g_j}g_{n-1}$ with $1\leq i=2k+1<2\varphi$ and $i+1<j\leq n-2$ is as follows:
\begin{enumerate}
\item\label{CCaa}
For $1\leq i=2k+1<2\varphi$ and $2\varphi+1\leq j\leq n-2$,
$$h(2,3)\cdots h(2k,2k+1)\,\,\&\,\,h_{j+1}\cdots h_{n-1}.$$

\item\label{CCbb}
For $1\leq i=2k+1<2\varphi$ and $i+1<j=2\ell\leq2\varphi$,
$$h(2,3)\cdots h(2k,2k+1)\,\,\&\,\,\overline{h_{2\ell}}h(2\ell,2\ell+1)\cdots h(2\varphi,2\varphi+1)h_{2\varphi+2}\cdots h_{n-1}.$$

\item\label{CCcc}
For $1\leq i=2k+1<2\varphi$ and $i+1<j=2\ell+1<2\varphi$,
$$h(2,3)\cdots h(2k,2k+1)\,\,\&\,\,h(2\ell+2,2\ell+3)\cdots h(2\varphi,2\varphi+1)h_{2\varphi+2}\cdots h_{n-1}.$$
\end{enumerate}

\item\label{DD} For $1\leq i=2k+1<2\varphi$, the $\phi$-image of $g_i\,\&\,g_{n-1}\overline{g_{i+1}}$ is
\small{$$h(2,3)\cdots h(2k,2k+1)\,\,\&\,\,h(2,3)\cdots h(2\varphi,2\varphi+1)
h_{2\varphi+2}\cdots h_{n-1}\overline{h(2,3)\cdots h(2k,2k+1)
h_{2k+2}}.$$}
\end{enumerate}
Note that the three expressions $h(2,3)\cdots h(2k-2,2k-1)h_{2k}$ in item (\ref{BB}) become $h_2$ when $k=1$, just as $h(2,3)\cdots h(2\varphi-2,2\varphi-1)h_{2\varphi}$ in item (\ref{AAbb}) becomes $h_2$ when $\varphi=1$. However, by definition, all four expressions $h(2,3)\cdots h(2k,2k+1)$ in items (\ref{CC}) and (\ref{DD}) must be interpreted as $h_1$ when $k=0$. These conventions will be in force throughout this and the next proofs.

\smallskip
Direct inspection shows that relations (\ref{artingenesis}) imply all relations in (\ref{AA})--(\ref{CC}). As for (\ref{DD}), we have
\begin{align}
h&(2,3)\cdots h(2\varphi,2\varphi+1)
h_{2\varphi+2}\cdots h_{n-1}\overline{h(2,3)\cdots h(2k,2k+1)
h_{2k+2}}\nonumber\\
&=h(2,3)\cdots h(2\varphi,2\varphi+1)
h_{2\varphi+2}\cdots h_{n-1}\overline{h_{2k+2}}\cdot\overline{h(2,3)\cdots h(2k,2k+1)}\nonumber\\
&=\left(h(2k+2,2k+3)\cdots h(2\varphi,2\varphi+1)
h_{2\varphi+2}\cdots h_{n-1}\overline{h_{2k+2}}\right)^{h(2,3)\cdots h(2k,2k+1)}\nonumber\\
&=\mbox{\small$\left(h(2k+2,2k+3)\cdot\overline{h_{2k+2}}\cdot h(2k+4,2k+5)\cdots h(2\varphi,2\varphi+1)
h_{2\varphi+2}\cdots h_{n-1}\right)^{h(2,3)\cdots h(2k,2k+1)}$}\nonumber\\
&=\left(h_{2k+3}\cdot h(2k+4,2k+5)\cdots h(2\varphi,2\varphi+1)
h_{2\varphi+2}\cdots h_{n-1}\right)^{h(2,3)\cdots h(2k,2k+1)}.\label{botondemuestra}
\end{align}
Here the second-to-last equality holds by assuming (\ref{artingenesis}), so that $h_{2k+2}$ commutes with every $h_i$ with $2k+4\leq i\leq n-1$. Likewise, the last equality follows by definition, as $h(2k+2,2k+3)\overline{h_{2k+2}}=h_{2k+3}$. Using (\ref{artingenesis}) again, we see that the base and the exponent in the conjugation-power (\ref{botondemuestra}) commute, leading to the commutation relation in (\ref{DD}).

The proof is complete by Proposition \ref{bestorganization} below.
\end{proof}

\begin{proposition}\label{bestorganization}
Assume $\varphi>0$ and that all relations (\ref{AA})--(\ref{DD}) hold. Then, for $v\in\{3,4,\ldots,n-1\}$, relations $h_u\,\&\,h_v$ hold for all $u\in\{1,2,\ldots,v-2\}$.
\end{proposition}
\begin{proof}
We prove the assertion by induction on $v$, first in the range $2\varphi+2\leq v\leq n-1$ and then, with a rather involved argument, in the range $3\leq v\leq 2\varphi+1$. The first inductive argument is a fairly direct generalization of the one in Example \ref{changen=p} (when $\varphi=0$): For $j=n-2$, relations (\ref{AA}), (\ref{BBaa}) and (\ref{CCaa}) show that $h_{n-1}$ commutes with
\begin{itemize}
\item $h(2,3)\cdots h(2\varphi,2\varphi+1)h_{2\varphi+2}\cdots h_i$, for $2\varphi+1\leq i\leq n-3$,
\item $h(2,3)\cdots h(2\varphi-2,2\varphi-1)h_{2\varphi}$,
\item $h(2,3)\cdots h(2k-2,2k-1)h_{2k}$, for $2\leq i=2k<2\varphi$, and
\item $h(2,3)\cdots h(2k,2k+1)$, for $1\leq i=2k+1<2\varphi$.
\end{itemize}
Reading these elements from the smallest to the largest values of $i$ (and using Lemma \ref{propiedadeselementales}), we get the assertion for $v=n-1$. This grounds the first inductive argument. The inductive step is formally identical: Take $v\in\{2\varphi+2,\ldots,n-2\}$ and assume that the assertion has been verified for all $v'$ with $v<v'\leq n-1$. Set $j:=v-1$ (so $j\geq2\varphi+1$). By induction (and Lemma \ref{propiedadeselementales}), relations (\ref{AA}), (\ref{BBaa}) and (\ref{CCaa}) imply that $h_{j+1}$ commutes with
\begin{itemize}
\item $h(2,3)\cdots h(2\varphi,2\varphi+1)h_{2\varphi+2}\cdots h_i$, for $2\varphi+1\leq i<j$ (this range is empty when $j=2\varphi+1$),
\item $h(2,3)\cdots h(2\varphi-2,2\varphi-1)h_{2\varphi}$,
\item $h(2,3)\cdots h(2k-2,2k-1)h_{2k}$, for $2\leq i=2k<2\varphi$,
\item $h(2,3)\cdots h(2k,2k+1)$, for $1\leq i=2k+1<2\varphi$.
\end{itemize}
As above, this closes the induction and proves the assertion for $v\in\{2\varphi+2,\ldots,n-1\}$. Next we prove the assertion for $v\in\{3,\ldots,2\varphi+1\}$ by inducting on $\lfloor v/2\rfloor$ ---rather than on~$v$. In other words, we consider the cases $v=2\omega$ and $v=2\omega+1$ ($1\leq\omega\leq \varphi$) simultaneously in the induction ---at the end of which (with $\omega=1$) we only consider the case $v=3$.

\smallskip
The induction starts with $v_\epsilon:=2\varphi+\epsilon$ for $\epsilon\in\{0,1\}$. Let $j_\epsilon:=v_\epsilon-1$. By Lemma \ref{propiedadeselementales} and the commutativity relations we have just proved, relations (\ref{BBbb}) for $j_1$ (i.e.~with $\ell=\varphi$) and (\ref{BBcc}) for $j_0$ (i.e.~with $\ell=\varphi-1$) imply that both $\overline{h_{2\varphi}}h(2\varphi,2\varphi+1)$ and $h(2\varphi,2\varphi+1)$ commute with $h(2,3)\cdots h(2k-2,2k-1)h_{2k}$ for $2\leq i=2k<2\varphi$. Equivalently (by Lemma~\ref{propiedadeselementales}),
$$
\mbox{$h_{2\varphi}$ and $h_{2\varphi+1}$ commute with $h(2,3)\cdots h(2k-2,2k-1)h_{2k}$ for $2\leq i=2k<2\varphi$.}
$$
Likewise, relations (\ref{CCbb}) for $j_1$ (i.e.~with $\ell=\varphi$) and (\ref{CCcc}) for $j_0$ (i.e.~with $\ell=\varphi-1$) give
$$
\mbox{$h_{2\varphi}$ and $h_{2\varphi+1}$ commute with
$h(2,3)\cdots h(2k,2k+1)$ for $1\leq i=2k+1<2\varphi-1$.}
$$
A new application of Lemma \ref{propiedadeselementales} then shows that
\begin{equation}\label{aux}
\mbox{$h_{2\varphi}$ and $h_{2\varphi+1}$ commute with $h_i$ for $1\leq i\leq2\varphi-2$.}
\end{equation}
Note that the last three assertions are vacuous when $\varphi=1$ and that, in order to complete the start of the induction, we still need to prove that $h_{2\varphi-1}\,\&\,h_{2\varphi+1}$ (which will also complete the whole induction when $\varphi=1$). With this in mind, take $i=2\varphi-1$ in (\ref{DD}) to get that $h(2,3)\cdots h(2\varphi-2,2\varphi-1)$ commutes with
\begin{align*}
h(2,3&)\cdots h(2\varphi,2\varphi+1)\cdot h_{2\varphi+2}\cdots h_{n-1}\cdot\overline{h(2,3)\cdots h(2\varphi-2,2\varphi-1)h_{2\varphi}}\\
=& \;h(2,3)\cdots h(2\varphi,2\varphi+1)\cdot h_{2\varphi+2}\cdots h_{n-1}\cdot\overline{h_{2\varphi}}\cdot\overline{h(2,3)\cdots h(2\varphi-2,2\varphi-1)}\\
=& \;h(2,3)\cdots h(2\varphi,2\varphi+1)\cdot\overline{h_{2\varphi}}
\cdot h_{2\varphi+2}\cdots h_{n-1}\cdot\overline{h(2,3)
\cdots h(2\varphi-2,2\varphi-1)} \\
=& \;h(2,3)\cdots h(2\varphi-2,2\varphi-1)\cdot h_{2\varphi+1}
\cdot h_{2\varphi+2}\cdots h_{n-1}\cdot
\overline{h(2,3)\cdots h(2\varphi-2,2\varphi-1)},
\end{align*}
where the second equality holds in view of what we proved in the first induction. Using Lemma~\ref{propiedadeselementales}, we then read the previous commutativity relation as the fact that 
$$
\mbox{$h(2,3)\cdots h(2\varphi-2,2\varphi-1)$ commutes with $h_{2\varphi+1}h_{2\varphi+2}\cdots h_{n-1}$ and thus with $h_{2\varphi+1}$,}
$$
which holds (once again) by the result of the first induction. Taking into account (\ref{aux}), we finally get the desired commutativity relation $h_{2\varphi-1}\,\&\,h_{2\varphi+1}$, thus grounding the second induction.

\smallskip
As with the first induction, settling the inductive step requires an argument entirely parallel to the one grounding the second induction. This time we only indicate where the two arguments have (minor) differences. 

\smallskip
The case $\varphi=1$ is covered by the start of the induction, so we can safely assume $\varphi\geq2$. Take $\omega\in\{1,2,\ldots,\varphi-1\}$ and assume that the assertion has been verified for $2\omega'$ and $2\omega'+1$ for all $\omega'\in\{\omega+1,\ldots,\varphi\}$. The focus then is on $v_\epsilon:=2\omega+\epsilon$ for $\epsilon\in\{0,1\}$. The exact same argument grounding the induction, but with $\varphi$ replaced by $\omega$, shows that $h_{2\omega}$ and $h_{2\omega+1}$ commute with $h_i$ for $1\leq i\leq2\omega-2$. Proving the additional relation $h_{2\omega-1}\,\&\,h_{2\omega+1}$ requires some tuning: Take $i=2\omega-1$ in (\ref{DD}) to get that $h(2,3)\cdots h(2\omega-2,2\omega-1)$ commutes with
\begin{align*}
h(2,3&)\cdots h(2\varphi,2\varphi+1)\cdot h_{2\varphi+2}\cdots h_{n-1}\cdot\overline{h(2,3)\cdots h(2\omega-2,2\omega-1)h_{2\omega}}\\
=& \;h(2,3)\cdots h(2\varphi,2\varphi+1)\cdot h_{2\varphi+2}\cdots h_{n-1}\cdot\overline{h_{2\omega}}\cdot\overline{h(2,3)\cdots h(2\omega-2,2\omega-1)}\\
=& \;h(2,3)\cdots h(2\omega,2\omega+1) \cdot \overline{h_{2\omega}} \cdot h(2\omega+2,2\omega+3)\cdots  h(2\varphi,2\varphi+1)\\
&\hspace{3mm}\cdot h_{2\varphi+2} \cdots h_{n-1} \cdot\overline{h(2,3)\cdots h(2\omega-2,2\omega-1)} \\
=& \;h(2,3)\cdots h(2\omega-2,2\omega-1)\cdot h_{2\omega+1}
\cdot h(2\omega+2,2\omega+3)\cdots h(2\varphi,2\varphi+1) \\
 &\hspace{3mm}\cdot h_{2\varphi+2}\cdots h_{n-1}\cdot
\overline{h(2,3)\cdots h(2\omega-2,2\omega-1)}.
\end{align*}
The exact same justifications as before allow us to deduce that $h(2,3)\cdots h(2\omega-2,2\omega-1)$ commutes with $h_{2\omega+1}\cdot h(2\omega+2,2\omega+3) \cdots h(2\varphi,2\varphi+1) \cdot\cdot h_{2\varphi+2}\cdots h_{n-1}$ and, then, with $h_{2\omega+1}$, which finally implies $h_{2\omega-1}\,\&\,h_{2\omega+1}$.
\end{proof}

We now prepare the grounds for the proof of Proposition \ref{encompassing}. 
\begin{notation}\label{palabras}\hspace{0em}{\em
\begin{enumerate}
\item Fix an integer $m$. Any element of $H_n$ that can be written as a word on generators $h_i$ with $i\leq m$ (respectively, $i\geq m$) and their inverses will generically be denoted as $w(m)$ (respectively, $W(m)$). Note that, for $m\leq0$, $w(m)$ can only denote the neutral element~$1$, while $W(m)$ would stand for any element of $H_n$. Likewise, for $m\geq n$, $W(m)=1$ while $w(m)$ would stand for any element of $H_n$. In some of the calculations below we need to deal with two potentially different generic elements of the same type and with the same parameter $m$. In such a case we will use $w'(x)$ or $W'(x)$, depending on the case, for the second generic element.

\item For integers $a_1,a_2,e_1,e_2,o_1,o_2\in\{1,2,\ldots,n-1\}$ with $e_1$, $e_2$ even and $o_1$, $o_2$ odd, set
\begin{align*}
A(a_1,a_2):=h_{a_1}h_{a_1+1}h_{a_1+2}\cdots h_{a_2},\\ E(e_1,e_2):=h_{e_1}h_{e_1+2}h_{e_1+4}\cdots h_{e_2},\\ O(o_1,o_2):=h_{o_1}h_{o_1+2}h_{o_1+4}\cdots h_{o_2}.
\end{align*}
Note that the order of factors in $E(e_1,e_2)$ or in $O(o_1,o_2)$ is immaterial, and that $E(e_1,e_2)=1$ (respectively, $O(o_1,o_2)=1$, $A(a_1,a_2)$=1) when $e_1>e_2$ (respectively $o_1>o_2$, $a_1>a_2$). In order to avoid awkward restrictions, it is convenient to extend the notation to unrestricted parameters. Thus, for arbitrary integers $a_1,a_2,e_1,e_2,o_1,o_2$ we set
\begin{align*}
&A(a_1,a_2):=A(\max(1,a_1),\min(n-1,a_2)), \\
&E(e_1,e_2):=E(\max(2,e_1),\min(2\lfloor(n-1)/2\rfloor,e_2)), \mbox{when both $e_1$ and $e_2$ are even,}\\
&O(o_1,o_2):=O(\max(1,o_1),\min(2\lceil(n-1)/2\rceil-1,o_2))), \mbox{when both $o_1$ and $o_2$ are odd.} 
\end{align*}
\end{enumerate}}\end{notation}

In terms of Notation \ref{palabras} (and in view of Proposition \ref{cambiodebasefinal}), the isomorphism $\phi\colon G_n\to H_n$ takes the more compact form 
$$
\phi(g_i)=\begin{cases}
h_1, & i=1;\\
O(3,2k-1)E(2,2k), & 2\leq i=2k\leq2\varphi;\\
O(3,2k+1)E(2,2k), & 3\leq i=2k+1\leq2\varphi+1;\\
O(3,2\varphi+1)E(2,2\varphi)A(2\varphi+2,i), & 2\varphi+2\leq i\leq n-1,
\end{cases}
$$ 
where the last instance can also be written as $O(3,2\varphi+1)A(2\varphi+2,i)E(2,2\varphi)$. Likewise, the following form of (\ref{furthernotice}) will be used without further notice in the rest of the section:
$$
\phi(\overline{g_i}g_{n-1})=\begin{cases}
A(i+1,n-1), & \mbox{for $i\geq\max(2,2\varphi+1)$;} \\
O(2\ell+3,2\varphi+1)E(2\ell+2,2\varphi)A(2\varphi+2,n-1), & \mbox{for $2\leq i=2\ell+1\leq2\varphi$;}\\
\overline{h_{2\ell}}\cdot O(2\ell+1,2\varphi+1)E(2\ell,2\varphi)A(2\varphi+2,n-1), & \mbox{for $2\leq i=2\ell\leq2\varphi$.}
\end{cases}
$$

The generic commutation relation $w(i)\,\&\,W(i+2)$, together with the generic expressions in Lemmas \ref{formulaschicas} and \ref{formulasgrandes} below, lay at the heart of our proof of Proposition \ref{encompassing} (given at the end of the section). 

\begin{lemma}\label{formulaschicas}
The following $w$-type generic expressions hold for non-negative integers $x$, $y$ and $z$ with $1\leq x\leq n-1\colon$
\begin{enumerate}
\item\label{gen0} $\phi(g_x)=w(x)$.

\item\label{gen1} $w(x)^{\phi(g_{n-1})}=\begin{cases}
w(x+1), & \mbox{if $x>2\varphi$, or if $x$ is even,} \\
w(x+2), & \mbox{otherwise.} 
\end{cases}$

\item\label{gen2} $\phi(g_x^{g_{n-1}^y})=w(x+y)$, if $x\geq2\varphi$.

\item\label{gen3} $\phi(g_{2\varphi-z}^{g_{n-1}^y})=w(2\varphi+y-\lfloor z/2\rfloor)$, if $1\leq z<2\varphi$.
\end{enumerate}
\end{lemma}
\begin{proof}
Item \ref{gen0} is obvious. For item \ref{gen1} with $x>2\varphi$:
\begin{align*}
w(x)^{\phi(g_{n-1})}&=w(x)^{O(3,2\varphi+1)E(2,2\varphi)A(2\varphi+2,n-1)}
=w(x)^{O(3,2\varphi+1)E(2,2\varphi)A(2\varphi+2,x+1)}\\
&=w(x+1).
\end{align*}
For item \ref{gen1} with $x=2\ell$, $1\leq\ell\leq\varphi$:
\begin{align*}
w(2\ell)^{\phi(g_{n-1})}&=w(2\ell)^{O(3,2\varphi+1)E(2,2\varphi)A(2\varphi+2,n-1)}=w(2\ell)^{O(3,2\varphi+1)E(2,2\ell)}\\
&=w(2\ell)^{O(3,2\ell+1)E(2,2\ell)O(2\ell+3,2\varphi+1)}
=w(2\ell)^{O(3,2\ell+1)E(2,2\ell)}\\&=w(2\ell+1).
\end{align*}
For item \ref{gen1} with $x=2\ell+1$, $0\leq\ell<\varphi$:
\begin{align*}
w(2\ell+1)^{\phi(g_{n-1})}&=w(2\ell+1)^{O(3,2\varphi+1)E(2,2\varphi)A(2\varphi+2,n-1)}=w(2\ell+1)^{O(3,2\varphi+1)E(2,2\ell+2)}\\
&=w(2\ell+1)^{O(3,2\ell+3)E(2,2\ell+2)O(2\ell+5,2\varphi+1)}
=w(2\ell+1)^{O(3,2\ell+3)E(2,2\ell+2)}\\&=w(2\ell+3).
\end{align*}
We leave for the reader the easy verification that items \ref{gen2} and \ref{gen3} follow from items \ref{gen0} and \ref{gen1}. (The generic estimate in item \ref{gen3} is not sharp when $y\leq\lfloor z/2\rfloor+1$, but this form is sufficient for our purposes.)
\end{proof}

\begin{lemma}\label{formulasgrandes}
The following $W$-type generic expressions hold for $1\leq x\leq n-1\colon$
\begin{enumerate}
\item\label{ric0} $\phi(\overline{g_x}g_{n-1})=\begin{cases}
W(x+1), & \mbox{if $x\geq\max(2,2\varphi+1)$, or if $x$ is odd with $x\geq3$,}\\
W(x), & \mbox{if $x\leq2\varphi$ and $x$ is even, or if $x=1$.}
\end{cases}$

\item\label{ric1} $W(x)^{\phi(g_{n-1})}=
\begin{cases}
W(2\varphi-1), &\mbox{if $x\geq2\varphi$,}\\
W(2\lfloor x/2\rfloor-1), &\mbox{if $x<2\varphi$.}
\end{cases}$

\item\label{ric2} $\phi((\overline{g_x}g_{n-1})^{g_{n-1}})=
\begin{cases}
W(2\varphi+1), & \mbox{if $x\geq2\varphi$,}\\
W(x), & \mbox{if $x<2\varphi$ and $x$ is odd,}\\
W(x+1), & \mbox{if $x<2\varphi$ and $x$ is even.}
\end{cases}$

\item\label{ric3} $\phi((\overline{g_x}g_{n-1})^{g_{n-1}^y})=W(2\varphi-2y+3)$, if $x\geq2\varphi$ and $y>0$. 
\end{enumerate}
\end{lemma}
\begin{proof}
Item \ref{ric0} is obvious. Item \ref{ric1} with $x\geq2\varphi$ is obvious for $\varphi\leq1$, whereas for $\varphi>1$:
\begin{align*}
W(x)^{\phi(g_{n-1})}&=W(x)^{O(3,2\varphi+1)A(2\varphi+2,n-1)E(2,2\varphi)}\\&
=W(x)^{O(3,2\varphi+1)A(2\varphi+2,n-1)h_{2\varphi}}\\
&=W(x)^{O(2\varphi-1,2\varphi+1)A(2\varphi+2,n-1)h_{2\varphi}O(3,2\varphi-3)}\\&=W(x)^{O(2\varphi-1,2\varphi+1) A(2\varphi+2,n-1)h_{2\varphi}}\\&=W(2\varphi-1).
\end{align*}
Item \ref{ric1} with $x=2\ell+\epsilon<2\varphi$ ($\epsilon\in\{0,1\}$) is elementary for $\ell\leq1$, while for $\ell\geq2$ (so $\varphi>2$):
\begin{align*}
W(2\ell+\epsilon)^{\phi(g_{n-1})}
&=W(2\ell+\epsilon)^{O(3,2\varphi+1)A(2\varphi+2,n-1)E(2,2\varphi)}\\&=W(2\ell+\epsilon)^{O(3,2\varphi+1)A(2\varphi+2,n-1)E(2\ell,2\varphi)}\\
&=W'(2\ell)^{O(3,2\varphi+1)}=W'(2\ell)^{O(2\ell-1,2\varphi+1)}\\&=W(2\ell-1).
\end{align*}
Item \ref{ric2} with $x>2\varphi$ is elementary for $\varphi=0$, while for $\varphi\geq1$ (so $x>2$):
\begin{align*}
\phi((\overline{g_x}g_{n-1})^{g_{n-1}})&=A(x+1,n-1)^{O(3,2\varphi+1) A(2\varphi+2,n-1)E(2,2\varphi)} \\& =A(x+1,n-1)^{O(3,2\varphi+1)A(2\varphi+2,n-1)}\\& =A(x+1,n-1)^{h_{2\varphi+1}A(2\varphi+2,n-1)O(3,2\varphi-1)}\\& =A(x+1,n-1)^{h_{2\varphi+1}A(2\varphi+2,n-1)}\\& =W(2\varphi+1).
\end{align*}
For item \ref{ric2} with $x=2\varphi$ (so $\varphi\geq1$ and $x\geq2$):
\begin{align*}
\phi((\overline{g_{2\varphi}}g_{n-1}&)^{g_{n-1}})=(\overline{h_{2\varphi}}\cdot h(2\varphi,2\varphi+1)A(2\varphi+2,n-1))^{O(3,2\varphi+1)A(2\varphi+2,n-1)E(2,2\varphi)}\\
&=(h(2\varphi,2\varphi+1)A(2\varphi+2,n-1)\cdot \overline{h_{2\varphi}})^{O(3,2\varphi+1)A(2\varphi+2,n-1)E(2,2\varphi-2)}\\
&=(h(2\varphi,2\varphi+1)\cdot\overline{h_{2\varphi}}\cdot A(2\varphi+2,n-1))^{O(3,2\varphi+1)A(2\varphi+2,n-1)E(2,2\varphi-2)}\\
&=A(2\varphi+1,n-1)^{O(3,2\varphi+1)A(2\varphi+2,n-1)E(2,2\varphi-2)}\\
&=A(2\varphi+1,n-1)^{O(3,2\varphi+1)A(2\varphi+2,n-1)}\\& =A(2\varphi+1,n-1)^{h_{2\varphi+1}A(2\varphi+2,n-1)O(3,2\varphi-1)}\\
&=A(2\varphi+1,n-1)^{h_{2\varphi+1}A(2\varphi+2,n-1)}=W(2\varphi+1).
\end{align*}
Item \ref{ric2} for $x=2\ell+1<2\varphi$ with $\ell=0$ is elementary, while for $\ell\geq1$ (so $x>2$):
\begin{align*}
\phi((&\overline{g_{2\ell+1}}g_{n-1})^{g_{n-1}})\\&
=(O(2\ell+3,2\varphi+1)E(2\ell+2,2\varphi)A(2\varphi+2,n-1)
)^{O(3,2\varphi+1)E(2,2\varphi)A(2\varphi+2,n-1)}\\
&=(A(2\varphi+2,n-1)O(2\ell+3,2\varphi+1)E(2\ell+2,2\varphi)
)^{O(3,2\varphi+1)E(2,2\varphi)}\\
&=(A(2\varphi+2,n-1)O(2\ell+3,2\varphi+1)E(2\ell+2,2\varphi)
)^{O(3,2\varphi+1)E(2\ell+2,2\varphi)}\\
&=(A(2\varphi+2,n-1)O(2\ell+3,2\varphi+1)E(2\ell+2,2\varphi)
)^{O(2\ell+1,2\varphi+1)E(2\ell+2,2\varphi)O(3,2\ell-1)}\\
&=(A(2\varphi+2,n-1)O(2\ell+3,2\varphi+1)E(2\ell+2,2\varphi)
)^{O(2\ell+1,2\varphi+1)E(2\ell+2,2\varphi)}=W(2\ell+1).
\end{align*}
For item \ref{ric2} with $x=2\ell<2\varphi$ (so $\varphi>\ell\geq1$ and $x\geq2$):
\begin{align*}
\phi(&(\overline{g_{2\ell}}g_{n-1})^{g_{n-1}})\\&
=(\overline{h_{2\ell}}\cdot O(2\ell+1,2\varphi+1)E(2\ell,2\varphi)A(2\varphi+2,n-1))^{O(3,2\varphi+1)E(2,2\varphi)A(2\varphi+2,n-1)}\\
&=(A(2\varphi+2,n-1)\cdot\overline{h_{2\ell}}\cdot O(2\ell+1,2\varphi+1)E(2\ell,2\varphi))^{O(3,2\varphi+1)E(2,2\varphi)}\\
&=(A(2\varphi+2,n-1)\cdot\overline{h_{2\ell}}\cdot O(2\ell+1,2\varphi+1)E(2\ell,2\varphi))^{O(3,2\varphi+1)E(2\ell,2\varphi)}\\
&=(\overline{h_{2\ell}}\cdot A(2\varphi+2,n-1) O(2\ell+1,2\varphi+1)E(2\ell,2\varphi))^{O(3,2\varphi+1)E(2\ell,2\varphi)}\\
&=(A(2\varphi+2,n-1)O(2\ell+1,2\varphi+1)E(2\ell+2,2\varphi))^{O(3,2\varphi+1)E(2\ell+2,2\varphi)}\\
&=(A(2\varphi+2,n-1)O(2\ell+1,2\varphi+1)E(2\ell+2,2\varphi))^{O(2\ell+1,2\varphi+1)E(2\ell+2,2\varphi)}=W(2\ell+1).
\end{align*}

We leave it for the reader the easy verification that item \ref{ric3} follows from items \ref{ric2} and \ref{ric1} ---in that order. 
\end{proof}

\begin{proof}[Proof of Proposition \ref{encompassing}]
Since $\pi\colon G_n\to B_n(p\times2)$ is a canonical projection, it suffices to show that the commutator coming from the two elements in each instance of (\ref{normalizadas}) ---subject to the conditions in (\ref{unifiedrelconds})--- vanishes in $G_n$. This will be achieved by showing that the image under the isomorphism $\phi\colon G_n\to H_n$ in Proposition~\ref{cambiodebasefinal} of each such commutator vanishes as, up to conjugates, it is of the form $[w(i),W(j)]$ with $i+2\leq j$. Explicitly, the commutator associated to (\ref{normalizadas}) is 
\begin{equation}\label{conmutador1}
\left[g_t\,,\,\overline{g_{s+t+u+1}}^{\hspace{.4mm}g_{n-1}^{m(t)-m(s+t+u+1)-s}}\cdot g_{n-1}\right],
\end{equation}
and its $g_{n-1}^{m(s+t+u+1)+s-m(t)}$-conjugate is
\begin{equation}\label{conmutador2}
\left[g_t^{\hspace{.4mm}g_{n-1}^{m(s+t+u+1)+s-m(t)}},\,\overline{g_{s+t+u+1}}\cdot g_{n-1}\right],
\end{equation}
where we recall that $m(j)=\max(0,\varphi-\lfloor j/2\rfloor)$. We then consider the image of (\ref{conmutador1}) under $\phi$ when $m(t)-m(s+t+u+1)-s\geq0$, and the $\phi$-image of (\ref{conmutador2}) when $m(s+t+u+1)+s-m(t)\geq0$. In either case, we simply use $A=A(r,s,t,u,v)$ and $B:=B(r,s,t,u,v)$, respectively, to refer to the first and second factors of the indicated commutator. The goal then is to show that, for each 5-tuple $(r,s,t,u,v)$ satisfying (\ref{unifiedrelconds}), there is a pair of integers $i$ and $j$ with $i+2\leq j$ and such that $\phi(A)=w(i)$ and $\phi(B)=W(j)$.

\smallskip
The simplest situation holds with $t\geq2\varphi$, as then $m(s+t+u+1)=m(t)=0$, so that
$\phi(A)=\phi(g_t^{g_{n-1}^s})=w(s+t=:i)$ while $\phi(B)=\phi(\overline{g_{s+t+u+1}}\cdot g_{n-1})=W(s+t+u+2=:j)$,
with $i+2\leq j$ obviously holding. Note that item 1 in Lemma \ref{formulasgrandes} applies in the second equality for $\phi(B)$ since $t>0$ in view of~(\ref{unifiedrelconds}). All other situations are similar and will be treated below on a case-by-case basis under the common hypothesis
$$
\mbox{$1\leq t=2\varphi-2\ell+\epsilon<2\varphi$, with $\epsilon\in\{0,1\}$ and $1\leq\ell\leq\varphi$, so $m(t)=\ell$.}
$$
Note that $\epsilon=1$ is forced when $\ell=\varphi$.

\medskip
\noindent{\bf Case $s\geq\ell$ with $s+u+\epsilon\geq2\ell$}. We have $s+t+u+1=2\varphi+s-2\ell+\epsilon+u+1\geq\max(2,2\varphi+1)$, so $m(s+t+u+1)=0$ and
$$
\phi(A)=\phi(g_{2\varphi-2\ell+\epsilon}^{g_{n-1}^{s-\ell}})=w(2\varphi+s-2\ell+\epsilon=:i),
$$
in view of Lemma \ref{formulaschicas}(\ref{gen3}), whereas Lemma \ref{formulasgrandes}(\ref{ric0}) gives
$$
\phi(B)=\phi(\overline{g_{2\varphi+s-2\ell+\epsilon+u+1}}\cdot g_{n-1})=W(2\varphi+s-2\ell+\epsilon+u+2=:j).
$$
The required condition $i+2\leq j$ is elementary.

\medskip\noindent{\bf Case $s\geq\ell$ with $s+u+\epsilon<2\ell$}. Say $s+u+\epsilon=2k-\delta$, with $\delta\in\{0,1\}$, so $\ell\leq2k-\delta<2\ell$, $1\leq k\leq\ell$, $s+t+u+1=2\varphi-2\ell+2k-\delta+1\leq2\varphi$ and $m(s+t+u+1)=\ell-k$. Then $m(s+t+u+1)+s-m(t)=s-k\geq s-\ell\geq0$, so that
$$
\phi(A)=\phi(g_{2\varphi-2\ell+\epsilon}^{g_{n-1}^{s-k}})=\begin{cases}
w(2\varphi-2\ell+\epsilon=:i), & \!\!\!\!\mbox{if $s=k$, by Lemma \ref{formulaschicas}(\ref{gen0});} \\
w(2\varphi-2\ell+2s-2k+2\epsilon-1=:i), & \!\!\!\!\mbox{if $s>k$, by Lemma \ref{formulaschicas}(\ref{gen0},\ref{gen1}),}
\end{cases}
$$
whereas Lemma \ref{formulasgrandes}(\ref{ric0}) gives
$$
\phi(B)=\phi(\overline{g_{2\varphi-2\ell+2k+1-\delta}}\cdot g_{n-1})=W(2\varphi-2\ell+2k+2-2\delta=:j).
$$
The required condition $i+2\leq j$ comes from the assumption $s+u+\epsilon=2k-\delta$.

\medskip\noindent{\bf Case $s<\ell$ with $s+u+\epsilon+1\geq2\ell$}. We have $s+t+u+1=2\varphi-2\ell+s+u+\epsilon+1\geq2\varphi$, so $m(s+t+u+1)=0$ and
$$
\phi(A)=\phi(g_t)=w(2\varphi-2\ell+\epsilon=:i),
$$
in view of Lemma \ref{formulaschicas}(\ref{gen0}), whereas Lemma \ref{formulasgrandes}(\ref{ric3}) gives
$$
\phi(B)=\phi(\overline{g_{s+t+u+1}}^{g_{n-1}^{\ell-s}}g_{n-1})=W(2\varphi-2\ell+2s+3=:j).
$$
The required condition $i+2\leq j$ is elementary.

\medskip In the only remaining case we have $s<\ell$ with $s+u+\epsilon+1<2\ell$. Say
\begin{equation}\label{ultima}
s+u+\epsilon+1=2k+\delta,
\end{equation}
with $\delta\in\{0,1\}$, so $s+t+u+1=2\varphi-2\ell+2k+\delta<2\varphi$, $m(s+t+u+1)=\ell-k>0$ and $m(t)-m(s+t+u+1)-s=k-s$. We then consider two subcases:

\medskip\noindent{\bf Subcase $k>s$}. Note that Lemma \ref{formulaschicas}(\ref{gen0}) gives $\phi(A)=\phi(g_t)=w(2\varphi-2\ell+\epsilon=:i)$, while Lemma~\ref{formulasgrandes}(\ref{ric1},\ref{ric2}) yields
$$
\phi(B)=\phi(\overline{g_{2\varphi-2\ell+2k+\delta}}^{g_{n-1}^{k-s}}g_{n-1})=W(2\varphi-2\ell+2s+3=:j).
$$
The required condition $i+2\leq j$ is elementary.

\medskip\noindent{\bf Subcase $k\leq s$}. We have
$$
\phi(A)=\phi(g_{2\varphi-2\ell+\epsilon}^{g_{n-1}^{s-k}})=\begin{cases}
w(2\varphi-2\ell+\epsilon=:i), & \!\!\!\!\mbox{if $s=k$, by Lemma \ref{formulaschicas}(\ref{gen0});} \\
w(2\varphi-2\ell+2s-2k+2\epsilon-1=:i), & \!\!\!\!\mbox{if $s>k$, by Lemma \ref{formulaschicas}(\ref{gen0},\ref{gen1}),}
\end{cases}
$$
whereas $\phi(B)=\phi(\overline{g_{2\varphi-2\ell+2k+\delta}}\cdot g_{n-1})=W(2\varphi-2\ell+2k+2\delta=:j)$, in view of Lemma~\ref{formulasgrandes}(\ref{ric0}) since $2\varphi-2\ell+2k+\delta\geq2$. (The latter inequality is obvious when $\varphi>\ell$, whereas if $\varphi=\ell$, so that $\varepsilon=1$, the inequality is a consequence of (\ref{ultima}).) The required condition $i+2\leq j$ is less obvious now and we close the proof by providing the easy argumentation.

\smallskip
When $s>k$, condition $i+2\leq j$ amounts to having $s+\epsilon+1\leq2k+\delta$, which is directly given by (\ref{ultima}). When $s=k$, condition $i+2\leq j$ amounts to having
\begin{equation}\label{check}
\epsilon+2\leq2k+2\delta,
\end{equation}
which obviously holds if $k\geq2$. If $k=1$, then (\ref{check}) could only fail with $\epsilon=1$ and $\delta=0$, which is ruled out by (\ref{ultima}). Lastly, if $k=0$, then (\ref{ultima}) forces $\delta=1$ and $s=u=\epsilon=0$, which yields (\ref{check}).
\end{proof}


\end{document}